\documentclass[reqno, 12pt]{amsart}
\pdfoutput=1
\makeatletter
\let\origsection=\section \def\section{\@ifstar{\origsection*}{\mysection}}
\def\mysection{\@startsection{section}{1}\z@{.7\linespacing\@plus\linespacing}{.5\linespacing}{\normalfont\scshape\centering\S}}
\makeatother

\usepackage{amsmath,amssymb,amsthm}
\usepackage{mathrsfs}
\usepackage{mathabx}\changenotsign
\usepackage{dsfont}

\usepackage{xcolor}
\usepackage[backref]{hyperref}
\hypersetup{
    colorlinks,
    linkcolor={red!60!black},
    citecolor={green!60!black},
    urlcolor={blue!60!black}
}

\usepackage[open,openlevel=2,atend]{bookmark}

\usepackage[abbrev,msc-links,backrefs]{amsrefs}
\usepackage{doi}

\renewcommand{\PrintDOI}[1]{\doi{#1}}

\usepackage[T1]{fontenc}
\usepackage{lmodern}
\usepackage[babel]{microtype}
\usepackage[english]{babel}

\usepackage{geometry}
\numberwithin{equation}{section}
\numberwithin{figure}{section}

\usepackage{enumitem}
\def\rmlabel{\upshape({\itshape \roman*\,})}

\def\alabel{\upshape({\itshape \alph*\,})}

\let\polishlcross=\l
\def\l{\ifmmode\ell\else\polishlcross\fi}

\def\qand{\quad\text{and}\quad}
\def\qqand{\qquad\text{and}\qquad}

\let\emptyset=\varnothing

\let\sm=\setminus

\makeatletter
\def\moverlay{\mathpalette\mov@rlay}
\def\mov@rlay#1#2{\leavevmode\vtop{   \baselineskip\z@skip \lineskiplimit-\maxdimen
   \ialign{\hfil$\m@th#1##$\hfil\cr#2\crcr}}}
\newcommand{\charfusion}[3][\mathord]{
    #1{\ifx#1\mathop\vphantom{#2}\fi
        \mathpalette\mov@rlay{#2\cr#3}
      }
    \ifx#1\mathop\expandafter\displaylimits\fi}
\makeatother

\newcommand{\dcup}{\charfusion[\mathbin]{\cup}{\cdot}}
\newcommand{\bigdcup}{\charfusion[\mathop]{\bigcup}{\cdot}}

\makeatletter
\DeclareFontFamily{U}  {MnSymbolC}{}
\DeclareSymbolFont{MnSyC}         {U}  {MnSymbolC}{m}{n}
\DeclareFontShape{U}{MnSymbolC}{m}{n}{
    <-6>  MnSymbolC5
   <6-7>  MnSymbolC6
   <7-8>  MnSymbolC7
   <8-9>  MnSymbolC8
   <9-10> MnSymbolC9
  <10-12> MnSymbolC10
  <12->   MnSymbolC12}{}
\DeclareMathSymbol{\powerset}{\mathord}{MnSyC}{180}

\DeclareFontFamily{U}{MnSymbolA}{}
\DeclareFontShape{U}{MnSymbolA}{m}{n}{
    <-6>  MnSymbolA5
   <6-7>  MnSymbolA6
   <7-8>  MnSymbolA7
   <8-9>  MnSymbolA8
   <9-10> MnSymbolA9
  <10-12> MnSymbolA10
  <12->   MnSymbolA12}{}
\DeclareFontShape{U}{MnSymbolA}{b}{n}{
    <-6>  MnSymbolA-Bold5
   <6-7>  MnSymbolA-Bold6
   <7-8>  MnSymbolA-Bold7
   <8-9>  MnSymbolA-Bold8
   <9-10> MnSymbolA-Bold9
  <10-12> MnSymbolA-Bold10
  <12->   MnSymbolA-Bold12}{}
\DeclareSymbolFont{MnSyA}{U}{MnSymbolA}{m}{n}
\SetSymbolFont{MnSyA}{bold}{U}{MnSymbolA}{b}{n}

\DeclareRobustCommand{\overleftharpoon}{\mathpalette{\overarrow@\leftharpoonfill@}}
\DeclareRobustCommand{\overrightharpoon}{\mathpalette{\overarrow@\rightharpoonfill@}}
\def\leftharpoonfill@{\arrowfill@\leftharpoondown\mn@relbar\mn@relbar}
\def\rightharpoonfill@{\arrowfill@\mn@relbar\mn@relbar\rightharpoonup}

\DeclareMathSymbol{\leftharpoondown}{\mathrel}{MnSyA}{'112}
\DeclareMathSymbol{\rightharpoonup}{\mathrel}{MnSyA}{'100}
\DeclareMathSymbol{\mn@relbar}{\mathrel}{MnSyA}{'320}

\makeatother

\usepackage{tikz}
\usetikzlibrary{calc,decorations.pathmorphing}
\pgfdeclarelayer{background}
\pgfdeclarelayer{foreground}
\pgfdeclarelayer{front}
\pgfsetlayers{background,main,foreground,front}

\let\epsilon=\varepsilon
\let\eps=\epsilon
\let\rho=\varrho
\let\theta=\vartheta

\let\phi=\varphi

\def\NN{{\mathds N}}

\def\ZZ{{\mathds Z}}
\def\PP{{\mathds P}}

\newcommand{\cK}{\mathcal{K}}

\newcommand{\cP}{\mathcal{P}}
\newcommand{\cQ}{\mathcal{Q}}

\newcommand{\cX}{\mathcal{X}}
\newcommand{\cY}{\mathcal{Y}}

\newcommand{\ccA}{\mathscr{A}}
\newcommand{\ccB}{\mathscr{B}}
\newcommand{\ccE}{\mathscr{E}}
\newcommand{\ccF}{\mathscr{F}}

\newcommand{\ccP}{\mathscr{P}}

\def\red{{\rm red}}
\def\blue{{\rm blue}}
\def\green{{\rm green}}

\theoremstyle{plain}
\newtheorem{thm}{Theorem}[section]

\newtheorem{prop}[thm]{Proposition}

\newtheorem{lemma}[thm]{Lemma}

\theoremstyle{definition}
\newtheorem{rem}[thm]{Remark}
\newtheorem{dfn}[thm]{Definition}
\newtheorem{exmp}[thm]{Example}
\newtheorem{conj}[thm]{Conjecture}

\usepackage{accents}

\DeclareMathOperator{\ex}{ex}

\usepackage{scalerel}

\newsavebox\vdegbox
\savebox\vdegbox{\tikz{
		\draw[black,fill=black] (90:1) circle (.35);
		\draw[black,line width=0.10cm] (210:1) circle (.30);
		\draw[black,line width=0.10cm] (330:1) circle (.30);
		\draw[opacity=0] (0:1.2) circle (0.1);
	}}

\newsavebox\vvbox
\savebox\vvbox{\tikz{
		\draw[black,line width=0.10cm] (90:1) circle (.30);
		\draw[black,fill=black] (210:1) circle (.35);
		\draw[black,fill=black] (330:1) circle (.35);
		\draw[opacity=0] (0:1.2) circle (0.1);
	}}

\newsavebox\pdegbox
\savebox\pdegbox{\tikz{
		\draw[black,line width=0.10cm] (90:1) circle (.30);
		\draw[black,fill=black] (210:1) circle (.35);
		\draw[black,fill=black] (330:1) circle (.35);
		\draw[black,line width=0.28cm ] (210:1) -- (330:1);
		\draw[opacity=0] (0:1.2) circle (0.1);
	}}

\newsavebox\vvvbox
\savebox\vvvbox{\tikz{
		\draw[black,fill=black] (90:1) circle (.35);
		\draw[black,fill=black] (210:1) circle (.35);
		\draw[black,fill=black] (330:1) circle (.35);
		\draw[opacity=0] (0:1.2) circle (0.1);
	}}
\newcommand{\vvv}{\mathord{\scaleobj{1.2}{\scalerel*{\usebox{\vvvbox}}{x}}}}
\newcommand{\pivvv}{\pi_{\vvv}}

\newsavebox\evbox
\savebox\evbox{\tikz{
		\draw[black,fill=black] (90:1) circle (.35);
		\draw[black,fill=black] (210:1) circle (.35);
		\draw[black,fill=black] (330:1) circle (.35);
		\draw[black,line width=0.28cm ] (210:1) -- (330:1);
		\draw[opacity=0] (0:1.2) circle (0.1);
	}}
\newcommand{\ev}{\mathord{\scaleobj{1.2}{\scalerel*{\usebox{\evbox}}{x}}}}	
\newcommand{\piev}{\pi_{\ev}}

\newsavebox\eebox
\savebox\eebox{\tikz{
		\draw[black,fill=black] (90:1) circle (.35);
		\draw[black,fill=black] (210:1) circle (.35);
		\draw[black,fill=black] (330:1) circle (.35);
		\draw[black,line width=0.28cm ] (90:1) -- (330:1);
		\draw[black,line width=0.28cm ] (90:1) -- (210:1);
		\draw[opacity=0] (0:1.2) circle (0.1);
	}}
\newcommand{\ee}{\mathord{\scaleobj{1.2}{\scalerel*{\usebox{\eebox}}{x}}}}
\newcommand{\piee}{\pi_{\ee}}

\newsavebox\eeebox
\savebox\eeebox{\tikz{
		\draw[black,fill=black] (90:1) circle (.35);
		\draw[black,fill=black] (210:1) circle (.35);
		\draw[black,fill=black] (330:1) circle (.35);
		\draw[black,line width=0.28cm ] (90:1) -- (330:1);
		\draw[black,line width=0.28cm ] (90:1) -- (210:1);
		\draw[black,line width=0.28cm ] (210:1) -- (330:1);
		\draw[opacity=0] (0:1.2) circle (0.1);
	}}
\newcommand{\eee}{\mathord{\scaleobj{1.2}{\scalerel*{\usebox{\eeebox}}{x}}}}
\newcommand{\pieee}{\pi_{\eee}}

\newcommand{\pird}{\pi^{\mathrm{rd}}}

\begin{document}
\title[Extremal problems in uniformly dense hypergraphs]{Extremal problems in uniformly dense hypergraphs}

\author[Christian Reiher]{Christian Reiher}
\address{Fachbereich Mathematik, Universit\"at Hamburg, Hamburg, Germany}
\email{Christian.Reiher@uni-hamburg.de}


\begin{abstract}
For a $k$-uniform hypergraph $F$ 
let $\textrm{ex}(n,F)$ be the maximum number of edges of a
$k$-uniform $n$-vertex hypergraph $H$ which contains no copy of~$F$.
Determining or estimating $\textrm{ex}(n,F)$ is a classical and central problem 
in extremal combinatorics. While for graphs ($k=2$) this problem is well understood, 
due to the work of Mantel, Tur\'an, Erd\H os, Stone, Simonovits and many others, only very little is known 
for $k$-uniform hypergraphs for $k>2$. Already the 
case when $F$ is a~$k$-uniform hypergraph with three edges on~$k+1$ vertices
is still wide open even for~$k=3$.

We consider variants of such problems where the large hypergraph~$H$ 
enjoys additional hereditary density conditions. Questions of this type were
suggested by Erd\H os and S\'os about 30 years ago. In recent work with 
R\"odl and Schacht it turned out that the \emph{regularity method for hypergraphs}, 
established by Gowers and by R\"odl et al.\ about a decade ago, is a suitable 
tool for extremal problems of this type and we shall discuss some of those 
recent results and some interesting open problems in this area.
\end{abstract} 

\maketitle


\section{Introduction}  
\subsection{Tur\'{a}n's extremal problem} 
\label{subsec:thp}
Extremal graph theory is known to have been initiated by Tur\'an's seminal article~\cite{Tu41},
in which he proved that for $n\ge r\ge 2$ there is, among all graphs on $n$ vertices not 
containing a clique of order $r$, exactly one whose number of edges is maximal, namely   
the balanced complete $(r-1)$-partite graph. 
Tur\'an then asked for similar results, where instead of a clique 
one intends to find the $1$-skeleton of a given platonic solid in the host graph. Moreover, 
he proposed to study analogous questions in the context of hypergraphs. 

Fixing some terminology, we say for a nonnegative integer $k$ that a pair $H=(V, E)$
is a {\it $k$-uniform hypergraph}, if~$V$ is a finite set of {\it vertices} and 
$E\subseteq V^{(k)}=\{e\subseteq V\colon |e|=k\}$ is a set of $k$-element subsets of~$V$, 
whose members are called the {\it edges} of $H$. As usual $2$-uniform hypergraphs are simply 
called {\it graphs}. Associated with every given $k$-uniform hypergraph~$F$ one has 
{\it Tur\'an's extremal function} $\ex(\cdot, F)$ mapping every positive integer~$n$ to 
\[
	\ex(n, F)
	=
	\max
	\big\{|E|\colon \text{$H=(V,E)$ is an $F$-free, $k$-uniform hypergraph with $|V|=n$}\big\}\,,
\]
i.e., to the largest number of edges 
that a $k$-uniform hypergraph on $n$ vertices 
without containing $F$ as a (not necessarily induced) subhypergraph can have. 
In its strictest sense, {\it Tur\'an's hypergraph problem} asks to determine this 
function for every hypergraph $F$. 

Using an averaging argument, Katona, Nemetz, and Simonovits~\cite{KNS64} have shown 
that for every $k$-uniform hypergraph $F$ the sequence 
$n\longmapsto \ex(n, F)\big/ \binom{n}{k}$ is nonincreasing. Therefore the limit
\[
	\pi(F)=\lim_{n\to\infty}\frac{\ex(n, F)}{\binom{n}{k}}\,,
\]
known as the {\it Tur\'{a}n density of $F$}, exists. The problem to determine the
Tur\'{a}n densities of all hypergraphs is likewise called {\it Tur\'{a}n's hypergraph
problem}.
 
It may be observed that these problems are trivial for $k\in\{0, 1\}$, while the case $k=2$ 
is fairly well understood. Tur\'{a}n himself~\cite{Tu41} determined $\ex(n, K_r)$ for all 
integers $n$ and~$r$, thus proving~$\pi(K_r)=\tfrac{r-2}{r-1}$ for every integer $r\ge 2$. 
This was further generalised by Erd\H{o}s and Stone~\cite{ErSt46}, and their result 
can be shown to yield the full answer to the 
Tur\'{a}n density problem in the case of graphs. Explicitly, we have
\begin{equation}\label{eq:miki}
	\pi(F)=\frac{\chi(F)-2}{\chi(F)-1} 
\end{equation}
for every graph $F$ with at least one edge, where $\chi(F)$ denotes the {\it chromatic
number} of~$F$, i.e., the least integer $r$ for which there exists a graph homomorphism
from $F$ to~$K_r$ (see also~\cite{ErSi66}, where the connection with the chromatic number 
appeared first).   

Already for $k=3$, however, our knowledge is very limited and there are only very few 
$3$-uniform hypergraphs $F$ for which the function $\ex(\cdot, F)$ is completely known. 
A notable example occurs when $F$ denotes the Fano plane. S\'os conjectured in the 1970s 
that for~$n\ge 7$ the balanced, complete, bipartite hypergraph is extremal for this problem.  
The first result in this direction is due to de Caen and F\"uredi~\cite{DeFu00}, 
who proved that at least the consequence $\pi(F)=\tfrac34$ of S\'os's conjecture holds. 
By combining their work with Simonovits's stability method~\cite{Si68} it was shown
in~\cites{FuSi05, KeSu05} that the conjecture holds for all sufficiently large 
hypergraphs. A full proof applying to all $n\ge 7$ was recently obtained in~\cite{BR}.

On the other hand, even concerning the $3$-uniform hypergraphs on four vertices with 
three and four edges, denoted by $K_4^{(3)-}$ and $K_4^{(3)}$ respectively, it is only known 
that
\[
		\frac{2}{7}\leq\pi(K_4^{(3)-})\leq 0.2871
		\qqand
		\frac{5}{9}\leq\pi(K_4^{(3)})\leq 0.5616\,.
\] 

The lower bounds are derived from explicit constructions due to 
Frankl and F\"uredi~\cite{FrFu84} and to Tur\'{a}n (see, e.g.,~\cite{Er77}),
and in both cases they are universally believed to be optimal. 
The upper bounds were obtained by computer assisted calculations based on Razborov's
\emph{flag algebra method} introduced in~\cite{Ra07}.
They are due to Baber and Talbot~\cite{BaTa11}, and to Razborov himself~\cite{Ra10}. 
For a more detailed discussion of our current knowledge about Turan's hypergraph problem 
we refer to Keevash's survey~\cite{Ke11}.

\subsection{Uniformly dense hypergraphs} 
\label{subsec:udh}
 
Let us now restrict our attention to $3$-uniform hypergraphs. Accordingly, the 
word {\it hypergraph} will henceforth always mean $3$-uniform hypergraph. Concerning the 
extremal problem for $K_4^{(3)-}$ it was thought for a while that its Tur\'an density 
might be $\tfrac14$. 

This notion was based on the following construction, which goes back to 
the work of Erd\H{o}s and Hajnal~\cite{ErHa72}. Take a random 
tournament $T$ on a large set~$V$ of vertices. Evidently any three vertices in $V$ induce 
either a transitive subtournament of~$T$ or a cyclic triangle. Furthermore, the former 
happens with a probability of $\tfrac34$ and the latter with a probability of~$\tfrac14$.
Define, depending on $T$, a random hypergraph $H(T)$ on $V$ whose edges correspond to the cyclic 
triangles in $T$. One checks easily that $H(T)$ can never contain a~$K_4^{(3)-}$ and the 
random choice of $T$ causes $H(T)$ to have, with high probability, an edge density close 
to~$\tfrac14$.

While the construction of Frankl and F\"uredi~\cite{FrFu84} mentioned earlier shows that 
the hypergraphs $H(T)$ cannot be {\it optimal among all $K_4^{(3)-}$-free hypergraphs},
it was suggested by Erd\H{o}s and S\'{o}s (see e.g.,~\cites{ErSo82, Er90}) that there 
might still be a natural sense in which they are optimal $K_4^{(3)-}$-free hypergraphs.
Specifically, they suggested to focus only on {\it uniformly dense} host hypergraphs defined
as follows. 

\begin{dfn} \label{dfn:vtxdense}
For real numbers $d\in[0, 1]$ and $\eta>0$ we say that a $3$-uniform hypergraph 
$H=(V, E)$ is {\it uniformly $(d, \eta)$-dense} if for all $U\subseteq V$ the 
estimate 
\[
	|U^{(3)}\cap E|\ge d\binom{|U|}{3}-\eta\,|V|^3
\]
holds. 
\end{dfn} 
 
Using standard probabilistic estimates one checks easily that for every accuracy 
para\-meter~${\eta>0}$ the probability that $H(T)$ is uniformly $(\tfrac14, \eta)$-dense 
tends to $1$ as the number of vertices tends to infinity. 
The Tur\'an theoretic question about the optimal 
density of uniformly dense hypergraphs not containing a given hypergraph $F$ (such as 
$K_4^{(3)-}$) can be made precise by introducing the quantities 
\begin{align} \label{eq:11}
	\pi_{\vvv}(F)  =\sup\bigl\{d\in[0,1]\colon &\text{for every $\eta>0$ and $n\in \NN$ there exists an $F$-free,} \notag\\
	& \text{ uniformly $(d,\eta)$-dense hypergraph $H$ with $|V(H)|\geq n$}\bigr\}\,,
\end{align}
which are to be regarded as modified versions of the usual Tur\'an densities
for uniformly dense hypergraphs. 
With this notation at hand,
the tournament construction shows that~$\pi_{\vvv}(K_4^{(3)-})\ge \tfrac 14$ 
and the aforementioned conjecture of Erd\H{o}s and S\'{o}s states that, actually,  
this holds with equality. Recently this has been shown independently in~\cite{GKV} 
and in~\cite{RRS-a}. 

\begin{thm} \label{thm:jems}
	We have $\pi_{\vvv}(K_4^{(3)-}) = \tfrac 14$.
\end{thm}

One of the two proofs referred to above consists of a 
computer-generated argument based on Razborov's flag algebra method, while the other one 
uses the hypergraph regularity method. 
The subsequent progress in this area (see~\cites{RRS-e, RRS-zero}) has followed the 
latter approach. Moreover, continuing the collaboration with R\"odl and Schacht, we have 
shown that there is a large number of further variants of the classical Tur\'an 
density that can likewise be studied by means of the hypergraph regularity method  
(see~\cites{RRS-c, RRS-d}). The goal of this article is to survey these 
recent developments.

Before we proceed any further, however, we would like to draw the reader's attention to 
perhaps the most urgent problem in the area, the determination of $\pivvv(K^{(3)}_4)$. The following 
construction, due to R\"odl~\cite{Ro86}, shows that this number has to be at least $\tfrac 12$.
Consider, for a sufficiently large natural number $n$, the elements of 
$[n]=\{1, 2, \ldots, n\}$ as vertices. Assign to every pair $ij$ of vertices uniformly 
at random one of the colours {\it red}
or {\it green}. Declare a triple $ijk$ with $1\le i<j<k\le n$ to be an edge of the 
hypergraph $H$ we are about to exhibit, if the colours of $ij$ and $ik$ disagree. 
Of course this happens with a probability of $\tfrac 12$ and, again, standard probabilistic 
arguments show that for every $\eta>0$ it happens asymptotically almost surely that $H$ 
is uniformly $(\frac 12, \eta)$-dense. Moreover, it is impossible that $H$ contains a 
tetrahedron. This is because for any four vertices $i<j<k<\ell$ it must be the case that two
of the three pairs $ij$, $ik$, and $i\ell$ receive the same colour, meaning that the three 
triples $ijk$, $ij\ell$, and $ik\ell$ cannot be present in $H$ at the same time. 
\begin{conj} \label{conj:wahres1/2}
	R\"odl's construction is optimal, i.e., we have $\pi_{\vvv}(K_4^{(3)})= \tfrac 12$.
\end{conj}

A partial result in this direction is given by Theorem~\ref{thm:jctb50} below. 

\subsection{Further Tur\'an densities} 
\label{subsec:13}
For proving results about $\pivvv(\cdot)$ one typically works with a property of 
hypergraphs that turns out to be more useful than the uniform density condition 
introduced in Definition~\ref{dfn:vtxdense}. Rather than knowing something about 
the edge densities {\it within single} sets of vertices, it is more helpful to 
have comparable knowledge about the edge densities {\it between any three} sets 
of vertices. Explicitly, if $H=(V, E)$ denotes a hypergraph and $A, B, C\subseteq V(H)$, 
we set 
\[
	E_{\vvv}(A, B, C)=\{(a, b, c)\in A\times B\times C\colon abc\in E\}\,.
\]
Moreover, for two real numbers $d\in[0, 1]$ and $\eta>0$ we say that $H$ 
is {\it $(d, \eta, \vvv)$-dense} if 
\[
	|E_{\vvv}(A, B, C)| \ge d |A| |B| |C| -\eta |V|^3
\]
holds for all $A, B, C\subseteq V$. One checks immediately by setting $U=A=B=C$ 
that every $(d, \eta, \vvv)$-dense hypergraph is also uniformly $(d, \eta/6)$-dense. 
In the converse direction one can only show that large uniformly dense hypergraphs 
contain linear sized subhypergraphs that are still dense in this new sense with almost
the same density, and that this is enough for proving   
\begin{align} \label{eq:12}
	\pi_{\vvv}(F) =\sup\bigl\{ d & \in[0,1]\colon  \text{for every $\eta>0$ and $n\in \NN$ there exists} \notag\\
	& \text{an $F$-free, $(d,\eta, \vvv)$-dense hypergraph $H$ with $|V(H)|\geq n$}\bigr\}\,.
\end{align}

A proof of this equality can be found in~\cite{RRS-e}*{Proposition~2.5}, where one has to 
set $k=3$ and $j=1$. Alternatively, the reader may prefer to regard~\eqref{eq:12} as the 
``official definition'' of~$\pivvv(\cdot)$ and treat~\eqref{eq:11} just like an 
additional piece of information that is not going to be used throughout the rest of this 
article. As a matter of fact, this may even be the more natural approach to this subject,
and the three dots occurring in the symbol~$\pivvv(\cdot)$ are intended to remind us of 
the three sets $A$, $B$, and $C$ mentioned in the definition of being~$(d, \eta, \vvv)$-dense.    
 
We proceed with a more restrictive property of hypergraphs shared by both the random 
tournament construction and by R\"odl's hypergraph introduced in the previous subsection. 
Given a hypergraph $H=(V, E)$, a set $A\subseteq V$, and a set of ordered 
pairs~$P\subseteq V^2$ we set 
\[
		E_{\ev}(A, P)=\{(a, b, c)\in V^3\colon a\in A, (b, c)\in P, \text{ and } abc\in E\}\,.
\]
So for instance $E_{\ev}(A, B\times C)=E_{\vvv}(A, B, C)$ holds for all $A, B, C\subseteq V$.
Next, for two real numbers $d\in[0, 1]$ and $\eta>0$ we say that $H$ is 
{\it $(d, \eta, \ev)$-dense} provided that 
\[
	|E_{\ev}(A, P)| \ge d |A| |P| -\eta |V|^3
\]
holds for all $A\subseteq V$ and $P\subseteq V^2$. Finally we define
\begin{align*} 
	\piev(F) =\sup\bigl\{d & \in[0,1]\colon  \text{for every $\eta>0$ and 
		$n\in \NN$ there exists} \notag\\
	& \text{an $F$-free, $(d,\eta, \ev)$-dense hypergraph $H$ with $|V(H)|\geq n$}\bigr\}
\end{align*}
for every hypergraph $F$. Since every $(d, \eta, \ev)$-dense hypergraph is, in particular, 
also $(d, \eta, \vvv)$-dense, we have 
\[
	\piev(F)\le \pivvv(F)
\]
for every hypergraph $F$. Let us remark at this point that due to the fact that 
R\"odl's hypergraph is $(\frac12, \eta, \ev)$-dense we have $\piev(K_4^{(3)})\ge\frac12$.
Thus the following result from~\cite{RRS-c} shows that a considerably weaker version of 
Conjecture~\ref{conj:wahres1/2} is true. 
\begin{thm} \label{thm:jctb50}
	We have $\piev(K_4^{(3)}) = \tfrac12$.
\end{thm}

The process of replacing a pair of sets by a set of pairs may be repeated once more. 
For a hypergraph $H=(V, E)$ and two sets of ordered pairs of vertices $P, Q\subseteq V^2$
one defines 
\[
	\cK_{\ee}(P, Q)=\{(a, b, c)\in V^3\colon (a, b)\in P \text{ and } (b, c)\in Q \}
\]
as well as 
\[
		E_{\ee}(P, Q)=\{(a, b, c)\in \cK_{\ee}(P, Q) \colon abc\in E\}\,.
\]
Notice that for all $A\subseteq V$ and $P\subseteq V^2$ we have 
\[ 
	|\cK_{\ee}(A\times V, P)|=|A| |P|
	\quad \text{ and } \quad 
	E_{\ee}(A\times V, P)=E_{\ev}(A, P)\,. 
\]
Next, we declare $H$ to be {\it $(d, \eta, \ee)$-dense}
for two real numbers $d\in [0, 1]$ and $\eta>0$ if
\[
	|E_{\ee}(P, Q)|\ge d |\cK_{\ee}(P, Q)| -\eta |V|^3 
\]
holds for all $P, Q\subseteq V^2$. If this is the case, then $H$ is $(d, \eta, \ev)$-dense 
as well. The generalised Tur\'an densities appropriate for this concept are defined by 
\begin{align*} 
	\piee(F) =\sup\bigl\{d & \in[0,1]\colon  \text{for every $\eta>0$ and $n\in \NN$ there exists} \notag\\
	& \text{an $F$-free, $(d,\eta, \ee)$-dense hypergraph $H$ with $|V(H)|\geq n$}\bigr\}\
\end{align*}
for every hypergraph $F$, and as before we may observe that
\[
	\piee(F) \le \piev(F)\,.
\]  

The investigation of these quantities was initiated in~\cite{RRS-d}, where the case that $F$
is a clique received particular attention. This led to the curious situation that while 
the value of $\piee(K^{(3)}_5)$ is still unknown, it has been be shown 
that $\piee(K^{(3)}_{11})=\tfrac 23$ holds (see Theorem~\ref{thm:fortress}). 
We would like to mention that $\ee$-dense 
hypergraphs have recently also been studied by Aigner-Horev and Levy~\cite{ELAD}
in the context of hypergraph Hamiltonicity problems. 

It is natural to expect at this moment some definitions of sets like $\cK_{\eee}(P, Q, R)$
and $E_{\eee}(P, Q, R)$ involving three sets of ordered pairs, but it can be shown that the 
corresponding generalised Tur\'an densities $\pieee(F)$ vanish for all hypergraphs $F$ 
(see~\cite{KRS02}). 

Still, there are some further variations on this theme. We refer to
the concluding remarks in~\cite{RRS-c} for a complete enumeration of all uniform density 
notions in the context of $3$-uniform hypergraphs\footnote{Strictly speaking, that article 
deals with quasirandomness notions instead of density notions, the difference being that 
in~\cite{RRS-c} there are also {\it upper bounds} imposed on the numbers $|E_{\vvv}(P, Q)|$, etc.
It seems, however, that the present version demanding only lower bounds on these numbers
is more natural from the perspective of  hypergraph Tur\'an problems.}. 
A more systematic account applying to $k$-uniform hypergraphs for all $k\ge 2$ has been 
given in~\cite{RRS-e}*{Section~2}. In this survey, however, we shall mainly focus 
on the most concrete cases~$\vvv$, $\ev$, and $\ee$.  
 
\section{Examples}

All known lower bounds on quantities of the form $\pi_\star(F)$ 
with $\star\in\{\vvv, \ev, \ee\}$ are derived from probabilistic constructions 
that can be viewed as appropriate modifications of R\"odl's hypergraph introduced at 
the end of Subsection~\ref{subsec:udh}. Basically, these constructions
combine an ordering of the vertices, a colouring of the pairs of vertices, 
and certain rules telling us which colour patterns on triples
of vertices are to be translated into edges of the envisioned hypergraph. 

As a matter of fact, even the Erd\H{o}s-Hajnal tournament hypergraph can be presented 
in this manner, even though prima facie it depends on an orientation rather than on a 
colouring of the pairs. 
Once its vertices receive an arbitrary ordering, however, there will be ``forward'' and 
``backward'' arcs between the vertices, and this state of affairs can alternatively be 
encoded by using two colours. Moreover, one can decide the presence or absence of an 
edge $abc$ in the hypergraph as soon as one knows the three ``colours'' received by the 
pairs $ab$, $ac$, and $bc$ (as well as the ordering of $\{a, b, c\}$).
 
For all these reasons, we shall now describe an abstract framework for presenting such 
constructions. Given a nonempty finite set $\Phi$ of {\it colours} we call a 
set $\ccP\subseteq \Phi^3$ a {\it palette} (over~$\Phi$). So the elements of 
palettes are ordered triples of colours, called {\it colour patterns}. Such a palette is said 
to be {\it $(d, \vvv)$-dense} for a real number $d\in [0, 1]$ if $|\ccP|\ge d|\Phi|^3$
holds. Given a vertex set $V$ equipped with a linear ordering $<$ and a 
colouring~${\phi\colon V^{(2)}\longrightarrow \Phi}$ we define a hypergraph $H^\ccP_\phi=(V, E)$
by 
\begin{equation}\label{eq:HphiP}
	E=\bigl\{\{x, y, z\}\in V^{(3)}\colon x<y<z \text{ and } 
		\bigl(\phi(x, y), \phi(x, z), \phi(y, z)\bigr)\in\ccP\bigr\}\,.
\end{equation}

In practice, one usually takes $V=[n]$ for a sufficiently large integer $n$ 
and adopts the standard ordering on this set as $<$. 
This causes no loss of generality in the sense that one still considers the same 
isomorphism types of hypergraphs as in the general case. 

Now the important observation is that if the underlying palette $\ccP$ is $(d, \vvv)$-dense
for some real $d\in [0, 1]$, and if the colouring $\phi$ gets chosen uniformly at random (among 
all~$|\Phi|^{\binom{n}{2}}$ possibilities), then for every $\eta>0$ the hypergraph $H^\ccP_\phi$
is asymptotically almost surely~$(d, \eta, \vvv)$-dense. Furthermore, given a hypergraph $F$ 
and a palette $\ccP$ it can be decided in finite time whether there exists a hypergraph 
of the form $H^\ccP_\phi$ containing $F$. Specifically, this happens if and only if there 
exists an ordering $<$ of $V(F)$ as well as a colouring 
$\phi\colon \partial F\longrightarrow \Phi$ of the set of pairs covered by edges of $F$
such that every edge $xyz$ of $F$ 
with $x < y < z$ satisfies
\[
	\bigl(\phi(x, y), \phi(x, z), \phi(y, z)\bigr)\in\ccP\,.
\]
Thus, whenever $F$ fails to admit such a pair $(<, \phi)$, one knows that $\pivvv(F)\ge d$. 

\begin{exmp}\label{exmp:null}
The simplest (nontrivial) palettes that can be imagined just consist of a single 
colour pattern. Owing to potential repetitions of colours in such a pattern, there
arise several distinct possibilities, the most restrictive of which is given 
by three distinct colours. So let us consider the case that $\Phi=\{\red, \blue, \green\}$ and 
$\ccP=\{(\red, \blue, \green)\}$. 

Clearly, $\ccP$ is $(\tfrac 1{27}, \vvv)$-dense. 
Therefore, the previous discussion shows that if a hypergraph~$F$ does not have 
property~\ref{zero:b} in Theorem~\ref{thm:zero} below, then $\pivvv(F)\ge \tfrac 1{27}$.
In other words, if $\pivvv(F)<\tfrac 1{27}$, then $F$ needs to admit such an ordering 
of its vertices together with 
such a colouring of its shadow. The main result of~\cite{RRS-zero} informs us that under this 
condition one actually has $\pivvv(F)=0$. This implies that $\pivvv(F)\not\in(0, \frac 1{27})$
holds for every hypergraph~$F$.

\end{exmp}
     
\begin{thm}\label{thm:zero}
For a $3$-uniform 
hypergraph $F$, the following are equivalent:
\begin{enumerate}[label=\alabel]
	\item \label{zero:a}$\pivvv(F)=0$.
	\item \label{zero:b} There is an enumeration of the vertex set $V(F)=\{v_1, \dots, v_f\}$ 
		and there is a three-colouring $\phi\colon \partial F\to\{\red,\blue,\green\}$ of the pairs of vertices covered by hyperedges of~$F$ 
		such that every hyperedge $\{v_i,v_j,v_k\}\in E(F)$ with $i<j<k$ satisfies 
		\[
			\phi(v_i,v_j)=\red,\quad \phi(v_i,v_k)=\blue, \qand \phi(v_j,v_k)=\green.
		\]
\end{enumerate}
\end{thm}

\begin{exmp} \label{ex:tourn}
As indicated by the discussion in the second paragraph of this section, the tournament 
hypergraph can be defined by the $(\tfrac 14, \vvv)$-dense palette 
\[
	\ccP=\{(\to, \leftarrow, \to), (\leftarrow, \to, \leftarrow)\}
\]
over $\Phi=\{\to, \leftarrow\}$. The proof of Theorem~\ref{thm:jems} presented in~\cite{RRS-a}
proceeds by showing that for $n^{-1} \ll \eta \ll \eps$ every 
$(\tfrac 14+\eps, \eta, \vvv)$-dense hypergraph on $n$ vertices possesses a vertex whose 
link graph contains a triangle. It thus seems natural to wonder whether similar ideas can 
be used to settle the value of $\pivvv(F)$ for all hypergraphs $F$ having a special vertex
contained in every edge. Given a graph $G$, let us call the hypergraph obtained from $G$
by adding a new vertex $\infty$ having all triples $\infty vw$ with $vw\in E(G)$ as edges 
the {\it cone over $G$}, denoted by~$CG$. So $K^{(3)-}_4 = CK_3$ and the question is what 
one can say about $\pivvv(CG)$ in general. This problem is already very interesting when $G$ 
is a clique. Concerning {\it stars} $S_k = CK_k$ 
the proof in~\cite{RRS-a} shows more generally that
\[
	\pivvv(S_k) \le \left(\frac{k-2}{k-1}\right)^2      
\]
holds for all $k\ge 2$, but it remains unclear at this moment whether this is sharp 
for any~$k\ge 4$. The $(\frac13, \vvv)$-dense palette 
\[
	\ccP= \{(1, 2, 1), (1, 3, 1), (2, 1, 2), (2, 3, 2), (3, 1, 3), (3, 2, 3), 
				(1, 2, 3), (2, 3, 1), (3, 1, 2)\}
\]
over $\Phi=\{1, 2, 3\}$ establishes the 
lower bound $\pivvv(S_4) \ge \frac 13$ and a generalisation of this idea leads
to
\[
	\pivvv(S_k)\ge \frac{k^2-5k+7}{(k-1)^2}
\]
for all $k\ge 3$ (see~\cite{RRS-a}*{Section~5.3.1}).
\end{exmp}

\begin{exmp} \label{ex:vojta}
R\"odl's hypergraph, let us recall, is defined by the $(\frac 12, \vvv)$-dense palette
\[
	\ccP=\{(\red, \green, \red), (\red, \green, \green), 
		(\green, \red, \red), (\green, \red, \green)\}
\]
over $\Phi=\{\red, \green\}$ and establishes $\pivvv(K^{(3)}_4)\ge\frac 12$. More generally, 
given a set $\Phi$ consisting of $r\ge 2$ colours one may use the palette 
\[
	\ccP=\{(\alpha, \beta, \gamma)\in \Phi^3\colon \alpha \ne \beta\}
\]
for showing 
\[
	\pivvv(K^{(3)}_{r+2}) \ge \frac{r-1}{r}\,.
\]

It would be exciting if equality turned out to hold here for all $r\ge 2$. It should be 
pointed out, however, that if this is true it might be much more difficult 
to prove than Conjecture~\ref{conj:wahres1/2}, as for $r=4$ there is a second, apparently 
sporadic, construction that yields the lower bound $\pivvv(K^{(3)}_{6})\ge \frac 34$
as well.
Namely, one takes the palette over $\{\red, \green\}$ containing all six colour 
patterns involving both colours (see~\cite{RRS-a}*{Section 5.1}). This construction
works because of $6\longrightarrow (3)^2_2$. However, we are probably just exploiting 
a numerical coincidence here and it seems unlikely that similar Ramsey theoretic 
constructions are relevant to the problem of determining $\pivvv(K^{(3)}_{r+2})$
(but see also Example~\ref{ex:ee11}).
\end{exmp}

\begin{exmp}\label{ex:cycle}
Finally, we briefly discuss the case where $F=C^{(3)}_5$ is a cycle of length five, 
i.e., $V(C^{(3)}_5)=\ZZ/5\ZZ$ and 
$E(C^{(3)}_5)=\big\{\{i, i+1, i+2\}\colon i\in \ZZ/5\ZZ\bigr\}$.
The lower bound $\pivvv(C^{(3)}_5)\ge \frac 4{27}$ can be shown by using the set of colours 
$\Phi=\{\rm{dark\, red}, \rm{light\, red}, \rm{green}\}$ and the palette consisting of all four 
colour patterns of the form $(\red, \red, \green)$, where ``$\red$'' means either 
``$\rm{dark\, red}$'' or ``$\rm{light\, red}$''. 
As far as we know, no interesting upper bound on $\pivvv(C^{(3)}_5)$
has ever been obtained. 
\end{exmp}

The last example suggests that occasionally it may be more convenient to work with a weighted
version of the concepts introduced so far. Let us say that a {\it weighted set of colours}
is a pair $(\Phi, w)$ consisting of a finite nonempty set of colours $\Phi$ and a weight 
function~$w\colon \Phi \to [0, 1]$ with the property $\sum_{\gamma\in \Phi}w(\gamma)=1$. 
If no weight 
function has been specified, we imagine that $w(\gamma)=|\Phi|^{-1}$ for all $\gamma\in\Phi$
is implicitly understood. Now when we have a palette $\ccP\subseteq \Phi^3$ over such a weighted
set of colours $(\Phi, w)$ we say that~$\ccP$ is {\it $(d, \vvv)$-dense} if 
$\sum_{(\alpha, \beta, \gamma)\in \ccP}w(\alpha)w(\beta)w(\gamma)\ge d$. In an obvious sense, 
this extends the meaning of being~{$(d, \vvv)$-dense} introduced earlier. Now instead of 
artificially talking about dark and light red in Example~\ref{ex:cycle} we could have just said
that we consider the weighted set $\Phi=\{\red, \green\}$ with $w(\red)=\frac 23$ and 
$w(\green)=\frac 13$, as well as the palette $\ccP=\{(\red, \red, \green)\}$. 

As long as the values attained by our weight function $w$ are rational numbers, 
it remains, of course, purely a matter of taste whether one prefers weighted sets of colours 
or whether one rather wants to speak about different shades of colours that are 
somewhat immaterial to the definition of the palette. It is an interesting open question, 
however, whether allowing irrational weights of the colours can ever give rise to an 
optimal lower bound on $\pivvv(F)$ for any hypergraph $F$. 

This roughly exhausts the lower bound constructions for $\pivvv(\cdot)$ that have been used 
so far, and we proceed with a discussion of $\piev(\cdot)$. Returning for simplicity to the 
unweighted setting, we say that a palette $\ccP$ over a set of colours $\Phi$ is 
{\it $(d, \ev)$-dense} for a real number~$d\in [0, 1]$ provided that 
\begin{center}
\begin{tabular}{c|c}
for every & we have  \\ \hline 
\rule{0pt}{3ex}  $\alpha\in\Phi$ &  $\big|\{(\beta, \gamma)\in \Phi^2
		\colon (\alpha, \beta, \gamma)\in \ccP\}\big|\ge d |\Phi|^2$\,, \\
\rule{0pt}{3ex}  $\beta\in\Phi$ &  $\big|\{(\alpha, \gamma)\in \Phi^2
		\colon (\alpha, \beta, \gamma)\in \ccP\}\big|\ge d |\Phi|^2$\,, \\
\rule{0pt}{3ex}  $\gamma\in\Phi$ &  $\big|\{(\alpha, \beta)\in \Phi^2
		\colon (\alpha, \beta, \gamma)\in \ccP\}\big|\ge d |\Phi|^2$\,. \\
\end{tabular}
\end{center}

\vskip.5em
Again easy probabilistic arguments show that whenever a palette $\ccP$ is $(d, \ev)$-dense,
and a colouring $\phi$ gets chosen uniformly at random, then for every $\eta>0$ the 
hypergraph $H^\ccP_\phi$ defined in~\eqref{eq:HphiP} is asymptotically almost surely 
$(d, \eta, \ev)$-dense. Thus lower bounds on $\piev(F)$ can be established almost 
in the same way as for $\pivvv(F)$, the only additional thing that needs to be checked 
being whether the palette one uses satisfies the three conditions in the above table. 

For instance, the palettes we referred to in the Examples~\ref{ex:tourn} and~\ref{ex:vojta}
are easily verified to be $\ev\,$-dense for the expected values of $d$. Hence the 
lower bounds on $\pivvv(\cdot)$ obtained there apply to $\piev(\cdot)$ as well. 
In particular, we learn         
\[
	\piev(S_k)\ge \frac{k^2-5k+7}{(k-1)^2}
\]
for every $k\ge 3$ and 
\[
	\piev(K^{(3)}_{r+2}) \ge \frac{r-1}{r}
\]
for every $r\ge 2$. But with the exception of Theorem~\ref{thm:jctb50} 
(and Theorem~\ref{thm:jems}) it is not known whether equality holds here either. 
The reader might briefly wonder at this point whether~$\pivvv(F)$ and $\piev(F)$ 
agree for all hypergraphs $F$. But in unpublished work with R\"{o}dl and Schacht 
it was shown that $\pivvv(F) >\piev(F)=0$ holds for some hypergraph $F$. 
Moreover, we obtained an explicit description of the class $\{F\colon \piev(F)=0\}$.
 
 The story of $\piee(\cdot)$ starts similarly, but the few results that have 
been obtained so far seem to suggest that this generalised Tur\'an density behaves
quite differently. To begin with, given a real number $d\in [0, 1]$ and a palette $\ccP$
over a set of colours $\Phi$, we say that $\ccP$ is {\it $(d, \ee)$-dense} if  
\begin{center}
\begin{tabular}{c|c}
for all & we have  \\ \hline 
\rule{0pt}{3ex}  $\alpha, \beta \in\Phi$ &  $\big|\{\gamma\in \Phi
		\colon (\alpha, \beta, \gamma)\in \ccP\}\big|\ge d |\Phi|$\,, \\
\rule{0pt}{3ex}  $\alpha, \gamma \in\Phi$ &  $\big|\{\beta\in \Phi
		\colon (\alpha, \beta, \gamma)\in \ccP\}\big|\ge d |\Phi|$\,, \\
\rule{0pt}{3ex}  $\beta, \gamma\in\Phi$ &  $\big|\{\alpha\in \Phi
		\colon (\alpha, \beta, \gamma)\in \ccP\}\big|\ge d |\Phi|$\,. \\
\end{tabular}
\end{center}

\vskip.5em
For clarity we emphasise that the two colours, $\alpha$, $\beta$, etc. mentioned 
in the left column of this table are allowed to be identical. Now again 
standard probabilistic arguments show that if $\ccP$ is $(d, \ee)$-dense, then for 
every $\eta>0$ the hypergraph $H^\ccP_\phi$ is asymptotically almost surely 
$(d, \eta, \ee)$-dense 
and this principle can be used in the standard way for producing lower bounds 
on $\piee(F)$ for many hypergraphs $F$. 

All palettes $\ccP$ used in this connection so far are {\it symmetrical}
in the sense that for every pattern $(\gamma_1, \gamma_2, \gamma_3)\in \ccP$ and every 
permutation $\sigma\in S_3$ one has 
$(\gamma_{\sigma(1)}, \gamma_{\sigma(2)}, \gamma_{\sigma(3)})\in \ccP$.
In other words, this means that permuting the entries of a triple does not affect 
its membership in the palette.  For symmetrical palettes any two of our three conditions 
are equivalent to each other, which reduces the amount of work one needs for checking them 
by a factor of three.\footnote[1]{It is for this reason that in~\cite{RRS-d}*{Section~13.1.3} 
only symmetrical palettes were introduced. Therefore, when writing~\cite{RRS-d}, it 
seemed more convenient to define palettes as collections of {\it multisets} of colours
instead of ordered triples, but it is unlikely that this will cause any confusion.}

When specifying a symmetrical palette, it is convenient to enumerate only a small
proportion of its colour patterns from which the remaining ones can be deduced owing
to the symmetry condition. More precisely, given an arbitrary palette $\ccP\subseteq \Phi^3$
we call the inclusion-wise minimal symmetrical palette containing $\ccP$ the {\it symmetrical 
palette generated by~$\ccP$}. One may observe that the three symmetrical palettes in the 
examples that follow possess some further symmetries induced by permutations of colours.   

\begin{exmp} \label{ex:ee5}
The symmetrical palette over $\{1, 2, 3\}$ generated by 
\[
	\{(1, 1, 2), (2, 2, 3), (3, 3, 1)\}
\]
is $(\frac 13, \ee)$-dense and shows $\piee(K^{(3)}_5)\ge \frac 13$ 
(see \cite{RRS-d}*{Section 13.1.3}).
\end{exmp}

\begin{exmp} \label{ex:ee6}
Similarly, the  symmetrical palette over $\{1, 2\}$ generated by $\{(1, 1, 2), (1, 2, 2)\}$
is $(\frac 12, \ee)$-dense and, due to the well-known Ramsey theoretic fact 
$6\longrightarrow (3)^2_2$, this proves that $\piee(K^{(3)}_6)\ge \frac 12$. 
\end{exmp}
 
\begin{exmp} \label{ex:ee11}
Finally, the  symmetrical palette over $\{1, 2, 3\}$ generated by 
\[
	\{(1, 1, 2), (1, 1, 3), (2, 2, 1), (2, 2, 3), (3, 3, 1), (3, 3, 2)\}
\]
is $(\frac 23, \ee)$-dense and because of a Ramsey theoretic result due to 
Chung and Graham~\cite{CG83} this proves $\piee(K^{(3)}_{11})\ge \frac 23$. 
\end{exmp}

The main result of~\cite{RRS-d} provides an upper bound on the $\ee\,$-Tur\'an-densities 
of cliques that turns out to be sharp in surprisingly many small cases. 

\begin{thm}\label{thm:fortress}
For every integer $r\ge 2$ one has $\piee(K_{2^r})\le \frac{r-2}{r-1}$.
\end{thm}

Together with the Examples~\ref{ex:ee5}--\ref{ex:ee11} this yields 
\begin{eqnarray*}
	\piee(K_4^{(3)}) = &0,&	\nonumber\\
	&\frac{1}{3}&\leq \piee(K_5^{(3)}),	\nonumber\\
	\piee(K_6^{(3)})
	=
	\piee(K_7^{(3)})
	=
	\piee(K_8^{(3)})
	=
	&\frac{1}{2}&
	\leq
	\piee(K_9^{(3)})
	\leq 
	\piee(K_{10}^{(3)}),	\\
	\piee(K_{11}^{(3)})
	=
	\dots
	=
	\piee(K_{16}^{(3)})
	=
	&\frac{2}{3}&\,,
	\nonumber
\end{eqnarray*}
i.e., the exact value of $\piee(K^{(3)}_t)$ for all $t \le 16$ with the exception 
of $t = 5, 9, 10$. It seems likely that if $\piee(K^{(3)}_5)=\tfrac 13$ turned 
out to be true, then the methods of~\cite{RRS-d} would allow to prove 
$\piee(K^{(3)}_{10})\le\tfrac 35$ as well. More generally, there are some good reasons to believe 
that $\piee(K^{(3)}_\ell)=\alpha$ implies $\piee(K^{(3)}_{2\ell})\le \frac 1{2-\alpha}$. 

\section{Reduced hypergraphs}

It is currently open whether all extremal hypergraphs for $\pivvv$, $\piev$, and $\piee$
can be derived from palettes, i.e., whether they are of the form $H^{\ccP}_\phi$. 
There is, however, a slightly more general method to construct
$(d, \eta, \star)$-dense hypergraphs with $\star\in\{\vvv, \ev, \ee\}$, for which such a result 
can be proved. This construction relies on so-called {\it reduced hypergraphs} that are going
to be introduced next. 

The main new idea is that when we have an ordered vertex set $(V, <)$ as well as a colouring 
$\phi$ of the pairs in $V^{(2)}$, then in hypergraphs of the form $H^{\ccP}_\phi=(V, E)$
the presence or absence of a triple $xyz$ with $x<y<z$ in $E$ depends entirely on the 
colours received by the pairs $xy$, $xz$, and $yz$ without taking the relative positions 
of $x$, $y$, and $z$ in the linear ordering $<$ into account. 
But one could imagine, for instance, hypergraphs with 
vertex set~$[2n]$ for some huge $n\in\NN$, where for converting colour patterns observed 
on pairs into edges there is one rule applying to triples with two vertices 
in $[n]$ and a completely different rule for triples with two vertices 
in $[n+1, 2n]$.

Reduced hypergraphs can be thought of as a framework for capturing the combinatorial core of 
all such constructions. Let us consider a finite set $I$ of {\it indices}. Suppose that 
to any pair of distinct indices $i, j\in I$ there has been assigned a finite nonempty set 
of vertices~$\cP^{ij}=\cP^{ji}$, and that for distinct pairs of indices the corresponding 
vertex sets are disjoint. 
Finally, assume that for every triple of indices $ijk\in I^{(3)}$ there has been specified 
a $3$-uniform tripartite hypergraph $\ccA^{ijk}$ with vertex classes $\cP^{ij}$, $\cP^{ik}$, 
and $\cP^{jk}$. In such situations we call the $\binom{|I|}{2}$-partite $3$-uniform hypergraph
$\ccA$ with 
\[
	V(\ccA)=\bigdcup_{ij\in I^{(2)}} \cP^{ij}
	\quad \text{ and } \quad 
	E(\ccA)=\bigdcup_{ijk\in I^{(3)}} E(\ccA^{ijk})
\]
a {\it reduced hypergraph} with {\it index set} $I$, {\it vertex classes} $\cP^{ij}$, 
and {\it constituents} $\ccA^{ijk}$. 

When translating such a reduced hypergraph into a hypergraph $H=(V, E)$ of Tur\'an 
theoretic significance one starts with a huge vertex set $V$ having an equipartition 
$V=\bigdcup_{i\in I}V_i$ and takes a ``colouring'' of pairs of vertices such that 
for $x\in V_i$ and $y\in V_j$ with $i\ne j$ the pair $xy$ receives uniformly at random 
some element of $\cP^{ij}$ as its ``colour'' $\phi(xy)$. Then for any three vertices 
from distinct partition classes $x\in V_i$, $y\in V_j$, and $z\in V_k$ one decides 
whether $xyz\in E$ should be the case depending on the colours $\phi(xy)$, $\phi(xz)$, 
and $\phi(yz)$ by using the constituent $\ccA^{ijk}$ as if it were a palette; 
so explicitly one demands
\[
	xyz\in E 
	\,\,\, \Longleftrightarrow \,\,\, 
	\{\phi(xy), \phi(xz), \phi(yz)\}\in E(\ccA^{ijk})\,.
\]

Next we need to express our density conditions in terms of reduced hypergraphs. 
The definition that follows is easy to remember. Intuitively it just tells us that 
$\vvv$, $\ev$, and $\ee$ correspond to ordinary density, a minimum vertex degree condition, 
and a minimum pair degree condition for the constituents of~$\ccA$, respectively.\footnote[1]{In 
the same way, the case $\eee$ dismissed at the the end of Section~\ref{subsec:13} 
would correspond to a minimum triple degree condition or, in other words, to the 
condition that all constituents be complete tripartite hypergraphs (if $d>0$). 
This is, of course, related to the fact that $\eee$-dense hypergraphs of positive density 
contain everything, i.e., that $\pieee(F)=0$ holds for every hypergraph $F$.}   

\begin{dfn}
	Let $\ccA$ denote a reduced hypergraph with index set $I$, vertex classes $\cP^{ij}$, 
	and constituents $\ccA^{ijk}$, and let $d\in [0, 1]$ be a real number.
	\begin{enumerate}[label=\rmlabel]
		\item\label{it:rdvvv} If $e(\ccA^{ijk})\ge d|\cP^{ij}||\cP^{ik}||\cP^{jk}|$
			holds for any three distinct indices $i, j, k\in I$ we say that $\ccA$
			is {\it $(d, \vvv)$-dense}. 
		\item\label{it:rdev} Moreover, if for any three distinct indices $i, j, k\in I$
			and every vertex $P^{ij}\in \cP^{ij}$ we have 
			\[
				\big|\bigl\{(P^{ik}, P^{jk})\in \cP^{ik}\times \cP^{jk}\colon
				(P^{ij}, P^{ik}, P^{jk})\in E(\ccA^{ijk})\bigr\}\big|
				\ge d |\cP^{ik}| |\cP^{jk}|\,,
			\]
			then $\ccA$ is called {\it $(d, \ev)$-dense}.
		\item\label{it:rdee} Finally, if for any three distinct indices $i, j, k\in I$
			and all vertices $P^{ij}\in \cP^{ij}$, $P^{ik}\in \cP^{ik}$ 
			we have 
			\[
				\big|\bigl\{P^{jk}\in  \cP^{jk}\colon
				(P^{ij}, P^{ik}, P^{jk})\in E(\ccA^{ijk})\bigr\}\big|
				\ge d |\cP^{jk}|\,,
			\]
			then $\ccA$ is called {\it $(d, \ee)$-dense}.
	\end{enumerate}
\end{dfn}
 
Whether the hypergraphs described by a given reduced hypergraph $\ccA$ are capable of 
containing a given hypergraph $F$ can be expressed in terms of the existence of so-called 
reduced maps, that are going to be introduced next.

\begin{dfn}\label{dfn:rm}
	A {\it reduced map} from a hypergraph $F$ to a reduced hypergraph $\ccA$ with 
	index set $I$, vertex classes $\cP^{ij}$, and constituents $\ccA^{ijk}$ is a pair 
	$(\lambda, \phi)$ such that 
	\begin{enumerate}[label=\rmlabel]
		\item\label{it:rm1} $\lambda\colon V(F)\longrightarrow I$  and  
			$\phi\colon \partial F\longrightarrow V(\ccA)$, where $\partial F$ denotes the 
				set of all pairs of vertices covered by an edge of $F$;
		\item\label{it:rm2} if $uv\in\partial F$, then $\lambda(u)\ne \lambda(v)$ and 
			$\phi(uv)\in\cP^{\lambda(u)\lambda(v)}$;
		\item\label{it:rm3} if $uvw\in E(F)$, then 
			$\{\phi(uv), \phi(uw), \phi(vw)\}
				\in E(\ccA^{\lambda(u)\lambda(v)\lambda(w)})$.
	\end{enumerate}
	
	If some such reduced map exists, we say that $\ccA$ {\it contains a reduced image 
	of $F$}, and otherwise $\ccA$ is called {\it $F$-free}. 
\end{dfn}

Now the main result about the reduced hypergraph construction asserts the following.

\begin{thm}\label{thm:33}
	If $F$ is a hypergraph and $\star\in\{\vvv, \ev, \ee\}$, then 
	\begin{align} \label{eq:pird}
		\pi_\star(F)  = \sup\bigl\{d\in [0, 1]\colon & \text{For every $m\in\NN$ } \text{there 
				is a $(d, \star)$-dense, } \notag \\
		& \text{$F$-free, reduced hypergraph with an index set of size $m$}\bigr\}\,. 
	\end{align}
\end{thm}

Large parts of the proof of this result are implicit in~\cites{RRS-a, RRS-c, RRS-d, RRS-zero}.
Still, we believe it to be useful to gather the argument in its entirety in the remainder of 
this section and the two subsequent sections. To this end, we shall temporarily denote the 
right side of~\eqref{eq:pird} by~$\pird_\star(F)$, where the superscript ``$\mathrm{rd}$'' 
means ``reduced''.

The inequality $\pird_\star(F) \le \pi_\star(F)$, proved in Proposition~\ref{prop:1}
below, simply expresses the fact that the narrative of this section does indeed indicate 
a valid strategy for establishing lower bounds on $\pi_\star(F)$ by means of 
reduced hypergraphs. The proof of the other direction, $\pird_\star(F) \ge \pi_\star(F)$, 
requires more involved reasoning based on the hypergraph regularity method. 
 
\begin{prop} \label{prop:1}
	For every hypergraph $F$ and every symbol $\star\in \{\vvv, \ev, \ee\}$ we have 
	\[
		\pird_\star(F) \le \pi_\star(F)\,.
	\]
\end{prop}

Let us recall the following standard concepts and facts required in the proof.
A bipartite graph $G=(X\dcup Y, E)$ is called {\it $(\delta, d)$-quasirandom}
for two real numbers $\delta>0$ and $d\in [0, 1]$ if for all $A\subseteq X$ and $B\subseteq Y$ 
the estimate $\big|e(A, B)-d|A||B|\big|\le \delta |X||Y|$ holds. Suppose now that for 
two nonempty disjoint sets $X$ and $Y$ we create a random bipartite graph $G$ with vertex set 
$X\dcup Y$ by declaring each pair in $K(X, Y)=\{\{x, y\}\colon x\in X \text{ and } y\in Y\}$ 
uniformly at random to be an edge of $G$ with probability $d$. Then for any fixed pair of sets 
$A\subseteq X$ and $B\subseteq Y$ Chernoff's inequality 
(see e.g.~\cite{AS}*{Theorem A.1.4}) implies
\[
	\PP\bigl(\big|e(A, B)-d|A||B|\big| > \delta |X||Y|\bigr) \le 2\exp(-2\delta^2|X||Y|), 
\]
whence 
\[
	\PP(\text{$G$ fails to be $(\delta, d)$-quasirandom})
	\le 
	2^{|X|+|Y|+1}\exp(-2\delta^2|X||Y|)\,.
\]
In particular, if $|X|=|Y|$ tends to infinity, then $G$ is asymptotically almost surely 
$(\delta, d)$-quasirandom. 

An important result about quasirandomness, utilised below,
is the so-called {\it triangle counting lemma}. It informs us that if a tripartite graph
$P=(X\dcup Y\dcup Z, E)$ has the property that its naturally induced bipartite subgraphs 
on $X\dcup Y$, $X\dcup Z$, and $Y\dcup Z$ are 
$(\delta, d_{XY})$-, $(\delta, d_{XZ})$-, and $(\delta, d_{YZ})$-quasirandom, respectively,
then the size of the set $\cK_3(P)=\{\{x, y, z\}\colon xy, xz, yz\in E\}$ of triangles it contains obeys the estimate
\[
	\big||\cK_3(P)|-d_{XY}d_{XZ}d_{YZ}|X||Y||Z|\big|\le 3\delta|X||Y||Z|\,.
\]

\begin{proof}[Proof of Proposition~\ref{prop:1}]
	Let a real number $d\in [0, 1]$ be given which has the property that for every 
	$m\in \NN$ there exists a $(d, \star)$-dense, $F$-free reduced hypergraph with 
	$m$ indices. We need to show that $d \le \pi_\star(F)$. So consider an 
	arbitrary real $\eta>0$ as well as some $n\in \NN$. Now we need to produce 
	a $(d, \eta, \star)$-dense, $F$-free hypergraph $H=(V, E)$ with $|V|\ge n$. 
	For this purpose, we set 
	\begin{equation}\label{eq:m}
		m=\left\lceil\frac 6\eta\right\rceil
	\end{equation}
	and appeal to our hypothesis on $d$. It yields an $F$-free, $(d, \star)$-dense 
	reduced hypergraph $\ccA$, 
	say with index set $I$, vertex classes $\cP^{ij}$, and constituents $\ccA^{ijk}$,
	where $|I| = m$. Now set 
	\[
		P=\max\bigl\{|\cP^{ij}|\colon ij\in I^{(2)}\bigr\}, 
		\quad 
		\delta=\frac{\eta}{6P^3}
	\]
	and let $h\gg P, m, n, \eta^{-1}$ be sufficiently large. 
	
	We shall construct the desired hypergraph $H$ on a set of vertices 
	$V=\bigdcup_{i\in I}V_i$ with $|V_i|=h$ for every $i\in I$. Owing to the probabilistic
	argument discussed before this proof we may assume that there is a family 
	$\{\phi_{ij}\colon ij\in I^{(2)}\}$ of colourings 
	$\phi_{ij}\colon K(V_i, V_j)\longrightarrow \cP^{ij}$ such that for every 
	pair of indices $ij\in I^{(2)}$ and every $P^{ij}\in \cP^{ij}$ the bipartite graph 
	$G(P^{ij})$ between $V_i$ and $V_j$ whose set of edges is $\phi^{-1}_{ij}(P^{ij})$
	happens to be $(\delta, |\cP^{ij}|^{-1})$-quasirandom.
	Depending on such colourings $\phi_{ij}$
	we complete the definition of $H$ in the expected way by setting 
	\begin{align*}
		E(H)=\bigl\{xyz\in V^{(3)}\colon \text{There are distinct } & i,j,k\in I  
			\text{ with $x\in V_i$, $y\in V_j$, $z\in V_k$,} \\
			& \text{and } \{\phi_{ij}(xy), \phi_{ik}(xz), \phi_{jk}(yz)\}\in E(\ccA^{ijk})
			\bigr\}\,.
	\end{align*}

	Let us remark that all edges of $H$ are {\it crossing} in the sense of intersecting 
	each of the vertex classes $V_i$ with $i\in I$ at most once. The rationale behind our
	choice of $m$ in~\eqref{eq:m} is that it allows us to bound the number of non-crossing
	triples $(x, y, z)\in V^3$ in a useful way. Clearly, this number is $h^3$ times the number 
	of triples $(i, j, k)\in I^3$ for which $i=j$, $i=k$, or $j=k$ holds. As this number 
	is in turn at most $3m^2$, we conclude that the number of non-crossing ordered triples 
	is at most $3m^2h^3 = 3m^{-1}|V|^3$, which by~\eqref{eq:m} is at most $\frac\eta 2|V|^3$.

	Now our choice of $h$ clearly guarantees $|V(H)|= hm\ge n$. 
	Next we would like to check that~$H$ is indeed $F$-free. 
	Otherwise there would exist an embedding $\psi\colon F\longrightarrow H$. For each 
	$u\in V(F)$ let $\lambda(u)\in I$ denote the index for which $\psi(u)\in V_{\lambda(u)}$
	is true. For every pair $uv\in\partial F$ we know that $\lambda(u)\ne \lambda(v)$, 
	because the edges of $H$ are crossing.
	Thus we may define $\phi\colon\partial F\to V(\ccA)$ by 
	\[
		\phi(uv)=\phi_{\lambda(u)\lambda(v)}\bigl(\psi(u)\psi(v)\bigr)
	\]
	for every 
	pair $uv\in \partial F$. Evidently~$\lambda$ and~$\phi$ satisfy the first two clauses 
	of Definition~\ref{dfn:rm}. As~$\psi$ maps edges of~$F$ to edges of~$H$, they 
	satisfy~\ref{it:rm3} as well. Thus $(\lambda, \phi)$ is a reduced map from $F$ to $\ccA$, 
	contrary to the choice of $\ccA$ as being $F$-free.
	
	It remains to check that $H$ is $(d, \eta, \star)$-dense and for this purpose 
	we consider the three possibilities for $\star$ separately. 
	
	\medskip
	
	{\hskip2em \it First Case. $\star=\vvv$}
	
	\smallskip
	
	Given arbitrary $A, B, C\subseteq V$ we need to prove that 
	$|E_{\vvv}(A, B, C)|\ge d |A||B||C|-\eta |V|^3$. 
	Whenever $i, j, k\in I$ are distinct and $\{P^{ij}, P^{ik}, P^{jk}\}\in E(\ccA^{ijk})$,
	the triangle counting lemma entails that the tripartite subgraph of 
	$G(P^{ij})\cup G(P^{ik})\cup G(P^{jk})$ induced by 
	$A\cap V_i$, $B\cap V_j$, and $C\cap V_k$ contains at least 
	\[
		\frac{|A\cap V_i||B\cap V_j||C\cap V_k|}{|\cP^{ij}||\cP^{ik}||\cP^{jk}|}
		-3\delta h^3
	\]
	triangles each of which gives rise to an edge of $H$. 
	Thus for distinct $i, j, k\in I$ we have 
	\[
		E_{\vvv}(A\cap V_i, B\cap V_j, C\cap V_k)
		\ge 
		\frac{|E(\ccA^{ijk})||A\cap V_i||B\cap V_j||C\cap V_k|}{|\cP^{ij}||\cP^{ik}||\cP^{jk}|}
		 		-3P^3\delta h^3\,,
	\]
	which by our assumption that $\ccA$ be $(d, \vvv)$-dense and by our choice of $\delta$
	yields
	\[
		E_{\vvv}(A\cap V_i, B\cap V_j, C\cap V_k)
		\ge 
		d|A\cap V_i||B\cap V_j||C\cap V_k|-\tfrac\eta 2 h^3\,.
	\]

	Summing over all ordered triples $(i, j, k)$ of distinct indices we infer that, up to 
	an additive error of at most $\frac\eta 2|V|^3$, the size of $E_{\vvv}(A, B, C)$ is at 
	least $d$ times the number of crossing triples in $A\times B\times C$. As there at most 
	$\frac\eta 2|V|^3$ non-crossing triples altogether, it follows that we have indeed 
	$|E_{\vvv}(A, B, C)|\ge d |A||B||C|-\eta |V|^3$.
	
	\goodbreak
	\medskip
	
	{\hskip2em \it Second Case. $\star=\ev$}
	
	\smallskip

	Given $A\subseteq V$ and $Q\subseteq V^2$ we need to prove that
	$|E_{\ev}(A, Q)|\ge d|A||Q| -\eta |V|^3$. Getting rid of non-crossing triples
	as in the previous case, it suffices to this end if we show for any three distinct indices 
	$i, j, k\in I$ that 
	\[
		  |E_{\ev}(A\cap V_i, Q\cap (V_j\times V_k)|
		  \ge d|A\cap V_i||Q\cap (V_j\times V_k)| - \tfrac\eta2 h^3\,.
	\]
	For this in turn it is enough to establish that for every $P^{jk}\in \cP^{jk}$
	the sets 
	\[
		\widetilde{K}(P^{jk})=\bigl\{(x, y, z)\in V_i\times V_j\times V_k\colon 
			x\in A, (y, z)\in Q, \text{ and } \phi_{jk}(yz)= P^{jk}\bigr\}
	\]
	and 
	\[
		\widetilde{E}(P^{jk})=\bigl\{(x, y, z)\in\widetilde{K}(P^{jk})\colon 
			xyz\in E(H)\bigr\}	 
	\]
	satisfy 
	\[
		\big|\widetilde{E}(P^{jk})\big| \ge d\, \big|\widetilde{K}(P^{jk})\big| 
			- \tfrac{\eta}{2P}h^3\,.
	\]

	Now we distinguish the triples $(x, y, z)$ contributing to the left side according 
	to the values of $\phi_{ij}(xy)$ and $\phi_{ik}(xz)$. By the assumed $(d, \ev)$-denseness
	of $\ccA$ we know 
	\[
		\big|\bigl\{(P^{ik}, P^{jk})\in \cP^{ik}\times \cP^{jk}\colon
			(P^{ij}, P^{ik}, P^{jk})\in E(\ccA^{ijk})\bigr\}\big|
		\ge 
		d |\cP^{ik}| |\cP^{jk}|\,,
	\]
	and thus it remains to show that for every edge $\{P^{ij}, P^{ik}, P^{jk}\}$ of 
	$\ccA^{ijk}$ we have 
	\begin{equation}\label{eq:33}
		\big|\bigl\{(x, y, z)\in \widetilde{K}(P^{jk})\colon 
		\phi_{ij}(xy)=P^{ij} \text{ and } \phi_{ik}(xz)=P^{ik} \bigr\}\big|
		\ge 
		\frac{\big|\widetilde{K}(P^{jk})\big|}{|\cP^{ij}||\cP^{ik}|}-3\delta h^3\,.
	\end{equation}

	Now appealing for $y\in V_j$ to the $(\delta, |\cP^{ik}|^{-1})$-quasirandomness 
	of $G(P^{ik})$ we learn that the sets 
	\[
		A_y=\bigl\{x\in A\cap V_i\colon \phi_{ij}(xy)=P^{ij}\bigr\}
	\]
	and 
	\[
		C_y=\bigl\{z\in V_k\colon (y, z)\in Q \text{ and } \phi_{jk}(yz)=P^{jk}\bigr\}
	\]
	satisfy
	\[
		\big|\bigl\{(x, z)\in A_y\times C_y\colon \phi_{ik}(xz)=P^{ik}\bigr\}\big|
		\ge 
		\frac{|A_y||C_y|}{|\cP^{ik}|}-\delta h^2\,.
	\]
	Summing over all $y\in Y$ we deduce 
	\begin{align}\label{eq:34}
		\big|\bigl\{(x, y, z)\in \widetilde{K}(P^{jk}) &\colon 
		\phi_{ij}(xy)=P^{ij} \text{ and } \phi_{ik}(xz)=P^{ik} \bigr\}\big| \notag \\
		& \ge 
		|\cP^{ik}|^{-1}
		\big|\bigl\{(x, y, z)\in \widetilde{K}(P^{jk})\colon 
		\phi_{ij}(xy)=P^{ij}\bigr\}\big| -\delta h^3\,.
	\end{align}

	Thus~\eqref{eq:33} will follow if can prove additionally that
	\[
		\big|\bigl\{(x, y, z)\in \widetilde{K}(P^{jk})\colon 
		\phi_{ij}(xy)=P^{ij}\bigr\}\big| 
		\ge 
		|\cP^{ij}|^{-1}\big|\widetilde{K}(P^{jk})\big|-\delta h^3\,.
	\]
	This estimate can be verified, however, in the same way as~\eqref{eq:34},
	the only difference being that this time one works with a sum over all $z\in V_k$
	and exploits the $(\delta, |\cP^{ij}|^{-1})$-quasirandomness of $G(P^{ij})$.	 
	
	\medskip
	
	{\hskip2em \it Third Case. $\star=\ee$}
	
	\smallskip
 	
	Proceeding almost exactly as in the previous case we consider two given sets of pairs
	$Q, R\in V^2$ and aiming at $|E_{\ee}(Q, R)|\ge d|\cK_{\ee}(Q, R)|-\eta |V|^3$ we begin 
	again by eliminating the noncrossing triples from our consideration, this time by reducing 
	our claim to the statement that for any three distinct indices $i, j, k\in I$ the 
	inequality 
	\[
		\big|E_{\ee}\bigl(Q\cap (V_i\times V_j), R\cap (V_j\times V_k)\bigr)\big|
		\ge 
		d\big|\cK_{\ee}\bigl(Q\cap (V_i\times V_j) , R\cap (V_j\times V_k)\bigr)\big|
			-\tfrac\eta 2h^3
	\]
	holds. This will be clear once we know that for all $P^{ij}\in\cP^{ij}$ and
	$P^{jk}\in\cP^{jk}$ the sets 
	\[
		\widetilde{K}(P^{ij}, P^{jk})
		=
		\bigl\{(x, y, z)\in V_i\times V_j\times V_k\cap \cK_{\ee}(Q, R) \colon
			\phi_{ij}(xy)=P^{ij} \text{ and } \phi_{jk}(yz)=P^{jk} \bigr\}
	\]
	and 
	\[
		\widetilde{E}(P^{ij}, P^{jk})
		=
		\bigl\{(x, y, z)\in \widetilde{K}(P^{ij}, P^{jk}) \colon
			xyz\in E(H) \bigr\}
	\]
	satisfy 
	\[
		\big|\widetilde{E}(P^{ij}, P^{jk})\big|
		\ge 
		d \big|\widetilde{K}(P^{ij}, P^{jk})\big|-3\delta Ph^3\,.
	\]
	As $\ccA$ is $(d, \ee)$-dense, we have 
	\[
		\big|\bigl\{P^{ik}\in  \cP^{ik}\colon
			(P^{ij}, P^{ik}, P^{jk})\in E(\ccA^{ijk})\bigr\}\big|
			\ge d |\cP^{ik}|\,,
	\]
	and, hence, it is enough to check that to every edge $\{P^{ij}, P^{ik}, P^{jk}\}$
	of the constituent $\ccA^{ijk}$ there corresponds an inequality 
	\begin{equation}\label{eq:35}
		\big|\bigl\{(x, y, z)\in \widetilde{K}(P^{ij}, P^{jk})\colon
			\phi_{ik}(xz)=P^{ik}\bigr\}\big|
		\ge \frac{\big|\widetilde{K}(P^{ij}, P^{jk})\big|}{|\cP^{ik}|}-\delta h^3\,.
	\end{equation}

	Now indeed for every $y\in V_j$ the $(\delta, |\cP^{ik}|^{-1})$-quasirandomness
	of $G(P^{ik})$ tells us that for the sets 
	\[
		A_y=\bigl\{x\in V_i\colon (x, y)\in Q \text{ and } \phi_{ij}(xy)=P^{ij}\bigr\}
	\]
	and 
	\[
		C_y=\bigl\{z\in V_k\colon (y, z)\in R \text{ and } \phi_{jk}(yz)=P^{jk}\bigr\}
	\]
	one has 
	\[
		\big|\bigl\{(x, z)\in A_y\times C_y\colon \phi_{ik}(xz)=P^{ik}\bigr\}\big|
		\ge
		\frac{|A_y||C_y|}{|\cP^{ik}|} - \delta h^2\,.
	\]
	By summing this over all $y\in Y$ one arrives at~\eqref{eq:35}.
\end{proof}

\section{Irregular triads}

The definition of $\pird_\star(F)$ assures us that for $\eps>0$ every 
$(\pird_\star(F)+\eps, \star)$-dense reduced hypergraph $\ccA$ with sufficiently many 
indices contains a reduced image of $F$. For our intended application of this fact, 
however, we need to know that it remains true if one allows the deletion of a small number
of edges from $\ccA$ (see Proposition~\ref{prop:irregular} below). 

For $\star=\vvv$ this turns out to be somewhat easier to prove 
than in the other two cases. The additional argument we want to
put forth if $\star=\ev$ is the following.   

Suppose that a reduced hypergraph $\ccA$ with index set $I$, 
vertex classes $\cP^{ij}$, and constituents $\ccA^{ijk}$ as well as two real numbers
$d>0$ and $\eta\ge 0$ are given. We shall say that~$\ccA$ is {\it $(d, \eta, \ev)$-dense} 
if for any three distinct indices $i, j, k\in I$ the exceptional set 
$\cX^{ij}_k\subseteq\cP^{ij}$ consisting of all $P^{ij}\in \cP^{ij}$ with 
\[
	\big|\bigl\{(P^{ik}, P^{jk})\in \cP^{ik}\times \cP^{jk}\colon
			(P^{ij}, P^{ik}, P^{jk})\in E(\ccA^{ijk})\bigr\}\big|
	< d |\cP^{ik}| |\cP^{jk}|
\]
satisfies $|\cX^{ij}_k|\le \eta |\cP^{ij}|$. So for reduced hypergraphs 
being $(d, 0, \ev)$-dense means the same as being $(d, \ev)$-dense. 

\begin{lemma}\label{lem:41}
	For every hypergraph $F$ and every $\eps >0$ there exist $m\in \NN$ and $\eta >0$ 
	such that every $(\pird_{\ev}(F)+\eps, \eta, \ev)$-dense reduced hypergraph 
	with $m$ indices contains a reduced image of $F$.
\end{lemma}  

\begin{proof}
	Choose $m$ in such a way that every $\bigl(\pird_{\ev}(F)+\frac\eps 2, \ev\bigr)$-dense 
	reduced hypergraph with~$m$ indices contains a reduced image of $F$, set   
	\[
		\eta=\frac{\eps}{4m}\,,
	\]
	and consider an arbitrary $(\pird_{\ev}(F)+\eps, \eta, \ev)$-dense reduced 
	hypergraph $\ccA$ with~$m$ indices. 
	As usual we denote the index set, vertex classes, and constituents of $\ccA$ by 
	$I$, $\cP^{ij}$, and $\ccA^{ijk}$ respectively. Let the exceptional sets $\cX^{ij}_k$
	be defined as above with $\pird_{\ev}(F)+\eps$ here in place of $d$ there. 
	
	Now the plan is to show that if one deletes all exceptional vertices from~$\ccA$
	one gets a $\bigl(\pird_{\ev}(F)+\frac\eps 2, \ev\bigr)$-dense reduced hypergraph, 
	which, therefore, contains a reduced image of $F$. Thus we define 
	\[
		\cQ^{ij}=\cP^{ij} \sm \bigcup_{k\in I\sm\{i, j\}}\cX^{ij}_k
		\quad \text{ for every pair } ij\in I^{(2)}\,,
	\]
	and notice that our assumption on~$\ccA$ implies 
	\begin{equation}\label{eq:Qij-large}
		|\cQ^{ij}|\ge (1-m\eta)\,,
	\end{equation}
	whence, in particular, $\cQ^{ij}\neq\varnothing$. 
	
	For this reason there exists a reduced hypergraph $\ccB$ with index set $I$
	and vertex classes~$\cQ^{ij}$ whose constituents $\ccB^{ijk}$ are the restrictions 
	of $\ccA^{ijk}$ to $\cQ^{ij}\dcup\cQ^{ik}\dcup\cQ^{jk}$.
	
	Consider any three distinct indices $i, j, k\in I$ as well as an 
	arbitrary vertex~${P^{ij}\in \cQ^{ij}}$. 
	From $P^{ij}\not\in\cX^{ij}_k$ we conclude 
	\begin{equation}\label{eq:41}
		\big|\bigl\{(P^{ik}, P^{jk})\in \cP^{ik}\times \cP^{jk}\colon
			(P^{ij}, P^{ik}, P^{jk})\in E(\ccA^{ijk})\bigr\}\big|
		\ge (\pird_{\ev}(F)+\eps) |\cP^{ik}| |\cP^{jk}|\,.
	\end{equation}

	By~\eqref{eq:Qij-large} we have ${|\cQ^{ik}|\ge (1-m\eta)|\cP^{ik}|}$
	and $|\cQ^{jk}|\ge (1-m\eta)|\cP^{jk}|$, wherefore
	\[
		|(\cP^{ik}\times \cP^{jk})\sm (\cQ^{ik}\times \cQ^{jk})|
		\le 
		2m\eta |\cP^{ik}||\cP^{jk}|
		=
		\tfrac\eps 2 |\cP^{ik}||\cP^{jk}|\,.
	\]
	Combined with~\eqref{eq:41} this yields 
	\begin{align*}
		\big|\bigl\{(P^{ik}, P^{jk})\in \cQ^{ik}\times \cQ^{jk}\colon
			(P^{ij}, P^{ik}, P^{jk})\in E(\ccA^{ijk})\bigr\}\big|
		&\ge 
		 \bigl(\pird_{\ev}(F)+\eps-\tfrac\eps 2\bigr) |\cP^{ik}| |\cP^{jk}| \\
		 &\ge 
		 \bigl(\pird_{\ev}(F)+\tfrac\eps 2\bigr) |\cQ^{ik}| |\cQ^{jk}| 
		 \,,
	\end{align*}
	as desired.
\end{proof}

Similar considerations can be undertaken with respect to $\ee$. As expected, 
a reduced hypergraph $\ccA$ with standard notation is called 
{\it $(d, \eta, \ee)$-dense} for two real numbers $d\in[0, 1]$ and $\eta>0$ 
provided that for any three distinct indices $i, j, k\in I$ the set $\cY^{jk}_i$
of all exceptional pairs $(P^{ij}, P^{ik})\in \cP^{ij}\times \cP^{ik}$ with 			%
\[
	\big|\bigl\{P^{jk}\in  \cP^{jk}\colon
	(P^{ij}, P^{ik}, P^{jk})\in E(\ccA^{ijk})\bigr\}\big|
	< d |\cP^{jk}|\,,
\]
satisfies $|\cY^{jk}_i|\le \eta |\cP^{ij}||\cP^{ik}|$. 

			
\begin{lemma}\label{lem:42}
	Given $\eps>0$ and a hypergraph $F$, there are $m\in \NN$ and $\eta>0$ such that 
	every $(\pird_{\ee}(F)+\eps, \eta, \ee)$-dense reduced hypergraph with $m$ 
	indices contains a reduced image of $F$.
\end{lemma}  

\begin{proof}
	Take $m$ so large that every $\bigl(\pird_{\ee}(F)+\tfrac\eps2, \ee\bigr)$-dense
	reduced hypergraph with $m$ indices contains a reduced image of $F$. Take $\ell\in\NN$
	and $\eta>0$ fitting into the hierarchy 
	\[
		\eta \ll \ell^{-1} \ll m^{-1}, \eps\,,
	\]
	and let $\ccA$ be a $(\pird_{\ee}(F)+\eps, \eta, \ee)$-dense reduced hypergraph
	with index set $I$ of size $m$, vertex classes~$\cP^{ij}$,  
	constituents $\ccA^{ijk}$, and exceptional sets $\cY^{jk}_i$ (defined with 
	$\pird_{\ee}(F)+\eps$ in place of $d$). 
	We want to prove that there is a reduced map from $F$ to $\ccA$. 
	
	To this end  we consider a random reduced hypergraph $\ccB$ with index set $I$ whose 
	vertex sets~$\cQ^{ij}$ are any $\binom{|I|}2$ disjoint sets of size $\ell$. 
	The intended randomness is induced by a family $\psi=\{\psi^{ij}\colon ij\in I^{(2)}\}$
	of maps $\psi^{ij}\colon \cQ^{ij}\longrightarrow \cP^{ij}$. Depending on $\psi$ 
	the constituents of $\ccB$ are defined so as to satisfy
	\[
		\{Q^{ij}, Q^{ik}, Q^{jk}\}\in E(\ccB^{ijk})
		\,\,\, \Longleftrightarrow \,\,\, 
 		\{\psi^{ij}(Q^{ij}), \psi^{ik}(Q^{ik}), \psi^{jk}(Q^{jk})\}\in E(\ccA^{ijk})
	\]
	for all $ijk\in I^{(3)}$ and all $Q^{ij}\in \cQ^{ij}$, $Q^{ik}\in\cQ^{ik}$, 
	and $Q^{jk}\in\cQ^{jk}$.	
	
	Let us observe first that if for some choice of $\psi$ it happens that
	$\ccB$ contains a reduced image of $F$, then we are done. This is because if 
	$(\lambda, \phi)$ is a reduced map from $F$ to $\ccB$, then $(\lambda, \psi\circ\phi)$ 
	is a reduced map from $F$ to $\ccA$,
	where by $(\psi\circ\phi)\colon \partial F\longrightarrow V(\ccA)$ we mean the map 
	defined by $(\psi\circ\phi)(uv)=\psi^{\lambda(u)\lambda(v)}\bigl(\phi(uv)\bigr)$ 
	for all $uv\in\partial F$.
	
	In the remainder of the proof we shall show that if $\psi$ gets chosen uniformly at 
	random, then with positive probability $\ccB$ 
	is $\bigl(\pird_{\ee}(F)+\tfrac\eps2, \ee\bigr)$-dense, which will conclude the 
	proof due to our choice of $m$. 
	
	So let us study for fixed distinct indices $i, j, k\in I$ and fixed vertices 
	$Q^{ij}\in\cQ^{ij}$, ${Q^{ik}\in\cQ^{jk}}$ the unpleasant event $\ccE$
	that the pair degree $D$ of $Q^{ij}$ and $Q^{ik}$ in $\ccB^{ijk}$ is smaller 
	than $\bigl(\pird_{\ee}(F)+\tfrac\eps2\bigr)\ell$.
	If also the vertices $P^{ij}=\psi^{ij}(Q^{ij})$ and $P^{ik}=\psi^{ik}(Q^{ik})$
	are given, then this pair degree $D$ depends only on $\psi^{jk}$ and not on the 
	remaining maps comprising $\psi$. Moreover, the distribution of $D$ is the same as
	if one draws $\ell$ random elements from $\cP^{jk}$ and keeps track of how many of them
	belong to the common neighbourhood of $P^{ij}$ and $P^{ik}$ in $\ccA^{ijk}$. Thus 
	if $(P^{ij}, P^{ik})\not\in \cY^{jk}_i$ the expected value of $D$ is at least 
	$(\pird_{\ee}(F)+\eps)\ell$ and Chernoff's inequality 
	(see e.g.~\cite{AS}*{Theorem A.1.4}) yields 
	\[
		\PP\bigl(
			\ccE\,
			\big| 
			\psi^{ij}(Q^{ij})=P^{ij} \text{ and } \psi^{ik}(Q^{ik})=P^{ik} 
			\bigr)
		\le
		\exp(-\eps^2\ell/2)\,. 
	\]
	Owing to $|\cY^{jk}_i|\le \eta |\cP^{ij}||\cP^{ik}|$ we infer
	\begin{align*}
		\PP(\ccE)
		&=
		|\cP^{ij}|^{-1}|\cP^{ik}|^{-1}
		\sum_{(P^{ij}, P^{ik})\in \cP^{ij}\times \cP^{ik}}
		\PP\bigl(
			\ccE\,
			\big| 
			\psi^{ij}(Q^{ij})=P^{ij}, \psi^{ik}(Q^{ik})=P^{ik} 
			\bigr) \\
		&\le \eta+\exp(-\eps^2\ell/2)\,. 
	\end{align*}
	As there are altogether no more than $m^3\ell^2$ possibilities to choose $i$, $j$, $k$, 
	$Q^{ij}$, and $Q^{ik}$, this proves 
	\[
		\PP\bigl(\ccB \text{ fails to be 
			$\bigl(\pird_{\ee}(F)+\tfrac\eps2, \ee\bigr)$-dense}\bigr)
		\le
		m^3\ell^2 \bigl(\eta+\exp(-\eps^2\ell/2)\bigr)\,,
	\]
	and by an appropriate choice of $\ell$ and $\eta$ the right side can be
	pushed below $1$.  
\end{proof}

\begin{rem}
	It should be clear that the same construction could have been used 
	for establishing Lemma~\ref{lem:41}. In fact, it generalises much further and 
	applies to the study of the Tur\'an densities $\pi_{\ccA}(F)$ initiated 
	in~\cite{RRS-e} as well. 
\end{rem}

\begin{prop}\label{prop:irregular} 
	Given a hypergraph $F$, a positive real number $\eps$, and a symbol 
	$\star\in\{\vvv, \ev, \ee\}$
	there exist $m\in \NN$ and $\delta >0$ such that the following holds. 
	If two reduced hypergraphs $\ccA$ and $\ccB$ with the same set of indices $I$
	of size at least $m$ and with the same vertex classes~$\cP^{ij}$ have the properties that 
	$\ccA$ is $(\pird_\star(F)+\eps, \star)$-dense and
	\begin{equation} \label{eq:AB}
		\sum_{ijk\in I^{(3)}} 
			\frac{|E(\ccA^{ijk})\sm E(\ccB^{ijk})|}{|\cP^{ij}||\cP^{ik}||\cP^{jk}|} 
		\le \delta |I|^3\,,
	\end{equation}
	then $\ccB$ contains a reduced image of $F$.  		
\end{prop}

\begin{proof}
	We work with the hierarchy
	\[
		\delta \ll \xi, m^{-1}\ll \eps\,.
	\]
	Call a triple $ijk\in I^{(3)}$ {\it useless} if 
	\[
		  E(\ccA^{ijk})\sm E(\ccB^{ijk})| > \xi |\cP^{ij}||\cP^{ik}||\cP^{jk}|\,.
	\]
	As a consequence of~\eqref{eq:AB} the number of such useless triples is at most 
	$\delta\xi^{-1}|I|^3$. Since~$|I|\ge m$ is sufficiently large, we have 
	$\binom{|I|}3 > \tfrac 17|I|^3$ and thus a proportion of no more than $7\delta\xi^{-1}$
	among all triples is useless. Therefore, if one draws a set $J\subseteq I$ with $|J|=m$
	uniformly at random, the expected number of useless triples in $J$ is at most 
	$7\delta\xi^{-1}\binom{m}3$, which by an appropriate choice of $\delta$ can be made  
	less than $1$.
	For this reason, there exists a set $J\subseteq I$ with $|J|=m$ spanning no useless 
	triple. We shall now prove that the restriction of $\ccB$ to $J$, denoted by $\ccB'$
	in the sequel, contains a reduced image of $F$. To this end we treat the three 
	cases $\star=\vvv, \,\ev,\, \ee$ separately. 
	
	\goodbreak
	\medskip
	
	{\it \hskip2em First Case. $\star=\vvv$}
	
	\smallskip
	
	Since $\xi\le \tfrac\eps2$, the reduced hypergraph $\ccB'$ is
	$\bigl(\pird_{\vvv}(F)+\tfrac\eps 2, \vvv\bigr)$-dense and thus it 
	does indeed contain a reduced image of $F$.
	
	\medskip
	
	{\it \hskip2em Second Case. $\star=\ev$}
	
	\smallskip
	
	Owing to Lemma~\ref{lem:41} it suffices to check that $\ccB'$ is 
	$\bigl(\pird_{\vvv}(F)+\tfrac\eps 2, 2\xi\eps^{-1}, \ev\bigr)$-dense.
	So let any three distinct indices $i, j, k\in J$ be given and let 
	$\cX^{ij}_k\subseteq\cP^{ij}$ denote the exceptional set of all vertices 
	$P^{ij}\in \cP^{ij}$ with 
	\[
		\big|\bigl\{(P^{ik}, P^{jk})\in \cP^{ik}\times \cP^{jk}\colon
				(P^{ij}, P^{ik}, P^{jk})\in E(\ccB^{ijk})\bigr\}\big|
		<
		\bigl(\pird_{\vvv}(F)+\tfrac\eps 2\bigr) |\cP^{ik}| |\cP^{jk}|\,.
	\]
	Since $\ccA$ is $(\pird_{\vvv}(F)+\eps, \ev)$-dense and $ijk$ is useful, 
	we have 
	\[
		|\cX^{ij}_k|\cdot \tfrac{\eps}2 	|\cP^{ik}| |\cP^{jk}|
		\le
		|E(\ccA^{ijk})\sm E(\ccB^{ijk})| 
		\le
		\xi |\cP^{ij}||\cP^{ik}||\cP^{jk}|\,,		
	\]
	which yields indeed $|\cX^{ij}_k|\le 2\xi\eps^{-1} |\cP^{ij}|$.

	\medskip
	
	{\it \hskip2em Third Case. $\star=\ee$}
	
	\smallskip
	
	Arguing as the previous case one proves that $\ccB'$ is 
	$\bigl(\pird_{\vvv}(F)+\tfrac\eps 2, 2\xi\eps^{-1}, \ee\bigr)$-dense,
	which  yields the desired conclusion in view of Lemma~\ref{lem:42}. 	
\end{proof}

\section{Hypergraph regularity}

The proof of Theorem~\ref{thm:33} can now be completed by means of the hypergraph 
regularity method, which for $3$-uniform hypergraphs is due to Frankl and R\"odl~\cite{FR}.
Our presentation below also takes the later works~\cites{RoSchRL, RoSchCL, Gow06, NPRS09}
into account. 

A central notion in this area is that of a hypergraph $H$ being regular with respect to 
a tripartite graph $P$, which roughly speaking means that the triangles in $P$ behave in an 
important way as if a random subset of them would correspond to edges of $H$.

\begin{dfn}
\label{def:reg}
A $3$-uniform hypergraph $H=(V,E_H)$ is {\it $(\delta_3,d_3)$-regular with respect to
a tripartite graph $P=(X\dcup Y\dcup Z,E_P)$} with $V\supseteq  X\cup Y\cup Z$ if 
for every tripartite subgraph~$Q\subseteq P$ we have 
\[
	\big||E_H\cap\cK_3(Q)|-d_3|\cK_3(Q)|\big|\leq \delta_3|\cK_3(P)|\,.
\]
Moreover, we simply say that {\it $H$ is $\delta_3$-regular with respect to $P$}, 
if it is $(\delta_3,d_3)$-regular for some $d_3\geq 0$.
We also define the {\it relative density} of $H$ with respect to $P$ by
\enlargethispage{2em}
\[
	d(H|P)=\frac{|E_H\cap\cK_3(P)|}{|\cK_3(P)|}\,,
\]
where we use the convention $d(H|P)=0$ if $\cK_3(P)=\emptyset$.
\end{dfn}

Now the hypergraph regularity lemma tells us that large hypergraphs can in the 
following approximate sense be decomposed into regular parts.

\begin{thm}[Regularity Lemma]
	\label{thm:RL}
	For every $\delta_3>0$, every $\delta_2\colon \NN \to (0,1]$, and every 
	$t_0\in\NN$ there exists an integer $T_0$ such that for every $n\geq t_0$
	and every $n$-vertex $3$-uniform hypergraph $H=(V,E_H)$ the following holds.
	
	There are integers $t\in[t_0, T_0]$ and $\ell \le T_0$,
	a vertex partition $V_0\dcup V_1\dcup\dots\dcup V_t=V$, 
	and for all $1\leq i<j\leq t$ there exists 
	a partition 
	\[
		\cP^{ij}=\{P^{ij}_\alpha=(V_i\dcup V_j,E^{ij}_\alpha)\colon\, 1\leq \alpha \leq \l\}
	\]
	of the edge set of the complete bipartite graph $K(V_i,V_j)$ satisfying the following 
	properties.
	\begin{enumerate}[label=\rmlabel]
		\item\label{it:RL1} $|V_0|\leq \delta_3 n$ and $|V_1|=\dots=|V_t|$;
		\item\label{it:RL2} for every $1\leq i<j\leq t$ and $\alpha\in [\l]$ the bipartite graph 
			$P^{ij}_\alpha$ is $(\delta_2(\l),1/\l)$-regular;
		\item\label{it:RL3} and $H$ is $\delta_3$-regular with respect to 
			$P^{ijk}_{\alpha\beta\gamma}$
			for all but at most $\delta_3t^3\l^3$ tripartite graphs 
			\begin{equation}\label{eq:triad}
				P^{ijk}_{\alpha\beta\gamma}=
				P^{ij}_\alpha\dcup P^{ik}_\beta\dcup P^{jk}_\gamma
				=(V_i\dcup V_j\dcup V_k, 
					E^{ij}_\alpha\dcup E^{ik}_{\beta}\dcup E^{jk}_{\gamma})
			\end{equation}
			with $1\leq i<j<k\leq t$ and $\alpha$, $\beta$, $\gamma\in[\l]$.
	\end{enumerate}
\end{thm}

The tripartite graphs occurring in~\eqref{eq:triad} are called {\it triads}.
In order to get a better feeling as to why (in our context) such a decomposition of 
a given hypergraph $H$
is a useful thing to have, it may be helpful to imagine the following special outcome. 
\begin{enumerate}[label=\alabel]
	\item\label{it:aa} $V_0=\varnothing$, i.e., the entire vertex set 
		gets partitioned;
	\item\label{it:bb} every edge of $H$ intersects each partition class $V_i$ at most once;
	\item\label{it:cc} there are no irregular triads, i.e.,~\ref{it:RL3} holds without any 
		exceptions;
	\item\label{it:dd} moreover, all triads are either ``full'' in the sense that all 
		their triangles 
		correspond to edges of $H$, or ``empty'' in the sense that none of their triangles 
		correspond to edges of $H$.   
\end{enumerate}

It is not hard to see that if these four things happen at the same time, then $H$ is 
essentially of the form constructed in the proof of Proposition~\ref{prop:1}. The underlying
reduced hypergraph~$\ccA$ on which such a construction would be based has index set $[t]$, 
vertex classes~$\cP^{ij}$, and the possible edges 
$\{P^{ij}_\alpha, P^{ik}_\beta, P^{jk}_\gamma\}$ in its constituents $\ccA^{ijk}$
would indicate which triads $P^{ijk}_{\alpha\beta\gamma}$ are ``full''. 
 
So in a vague sense what remains to be done for completing the proof of Theorem~\ref{thm:33}
is that we need to address how to deal with the possible failures of~\ref{it:aa}--\ref{it:dd}
when the regularity lemma gets applied. There will be no difficulties with~\ref{it:aa} 
or~\ref{it:bb}, for the concepts we study are sufficiently robust, so that deleting the 
small set $V_0$ for~\ref{it:aa} and ignoring the small proportion of noncrossing edges
for~\ref{it:bb} has essentially no effect. We are prepared for~\ref{it:cc} in the light of 
Proposition~\ref{prop:irregular}. 

Finally, regarding~\ref{it:dd} we will treat triads with respect 
to which the relative density $H$ is not too small as if they were full. 
That is, for some appropriate constant $d_3>0$ we will put an edge 
$\{P^{ij}_\alpha, P^{ik}_\beta, P^{jk}_\gamma\}$ into $\ccA^{ijk}$ if and only if 
$d(H|P^{ijk}_{\alpha\beta\gamma})\ge d_3$. This will allow us to rather easily transfer 
denseness properties from $H$ to $\ccA$, but we will need an argument as to why 
a reduced map from $F$ to $\ccA$ does still give rise to a copy of $F$ in $H$, 
even though the triads we want to use are not known to be full. This is, however,
a standard situation in hypergraph regularity theory, for which the {\it counting 
lemma} has been developed. Below we shall require the following consequence of this result.   
   
\begin{thm}[Embedding Lemma]\label{thm:EL}
	For every $3$-uniform hypergraph $F$ and every $d_3>0$ 
	there exist $\delta_3>0$, and functions 
	$\delta_2\colon \NN\longrightarrow(0,1]$ and  $N\colon \NN\longrightarrow\NN$
	such that the following holds for every $\l\in\NN$.
	
	Suppose that 
	\begin{enumerate} 
		\item[$\bullet$] $\lambda\colon V(F)\longrightarrow I$ is a map 
			from $V(F)$ to some set $I$ with $\lambda(u)\ne \lambda(v)$
			 for all $uv\in\partial F$,
		\item[$\bullet$] that $\{V_i\colon i\in I\}$ is a family of mutually disjoint 
			sets of the same size $N_*\ge N(\ell)$,
		\item[$\bullet$] and that for every $uv\in \partial F$ one has a 
			 $(\delta_2(\l),\frac{1}{\l})$-quasirandom bipartite graph
			 $P_{uv}$ between $V_{\lambda(u)}$ and $V_{\lambda(v)}$.
	\end{enumerate}
	
	Then a hypergraph $H$ with $V(H) \supseteq \bigcup_{i\in I}V_i$ posseses a 
	subhypergraph isomorphic to $F$ provided that for every edge $uvw\in E(F)$ 
	\begin{enumerate}
		\item[$\bullet$] one has 
			$d(H|P_{uv}\cup P_{uw}\cup P_{vw})\geq d_3$
		\item[$\bullet$] and $H$ is $\delta_3$-regular with respect to the tripartite 
			graph $P_{uv}\cup P_{uw}\cup P_{vw}$.
	\end{enumerate}
\end{thm}

For completeness we shall briefly discuss how this statement relates to the 
standard reference~\cite{NPRS09}*{Corollary~2.3}. First of all, a more conventional 
setup for the counting lemma would be the case that $V(F)=I=[f]$ holds for some natural 
number $f$ and that $\lambda$ is the identity. Secondly, in this special case the full 
counting lemma allows to estimate the number of homomorphisms $\psi$ 
from $F$ to $H$ with $\psi(u)\in V_{\lambda(u)}$ for every $u\in V(F)$ in a satisfactory way. 
In particular, a suitable choice of $\delta_3$, $\delta_2(\cdot)$,
and $N(\cdot)$ entails that this number is at least 
$\frac12 d_3^{e(F)}\ell^{-|\partial F|}N_*^{|V(F)|}$. 
Thirdly, this assertion generalises immediately to the case of general $F$, $I$, and $\lambda$, 
even if $\lambda$ should fail to be injective. Finally, by increasing $N(\ell)$ if necessary, 
one can achieve that this lower bound on the number of homomorphisms from~$F$ to $H$ exceeds
the number of non-injective maps $\psi$ from $V(F)$ to $V(H)$ 
with $\psi(u)\in V_{\lambda(u)}$ for every $u\in V(F)$. Therefore,~\cite{NPRS09}*{Corollary~2.3}
does indeed imply Theorem~\ref{thm:EL}. 

We may now proceed to the second half of Theorem~\ref{thm:33}.

\begin{prop}\label{prop:3}
	If $F$ is a hypergraph and $\star\in\{\vvv, \ev, \ee\}$, then 
	$\pi_\star(F)\le \pird_\star(F)$.
\end{prop}

\begin{proof}
	We may suppose $\pird_\star(F)<1$, since otherwise the result is clear.   
	Let an arbitrary $\eps\in (0, 1-\pird_\star(F)]$ be given. 
	By plugging $F$ and $d_3=\frac17\eps$ into 
	Theorem~\ref{thm:EL} we obtain $\delta_3>0$ and functions 
	$\delta_2\colon \NN\longrightarrow (0, 1]$ as well as 
	$N\colon \NN\longrightarrow \NN$. Without loss of generality, we may suppose that 
	$\delta_3<\tfrac12$ is sufficiently small, 
	that $\delta_2(\ell)\le \frac1{21}\eps\ell^{-3}$ holds for every $\ell\in \NN$ 
	and that~$N$ is increasing.
	By Proposition~\ref{prop:irregular} and our flexibility to decrease $\delta_3$  
	we may assume that there exists $t_0\in \NN$ such that if for arbitrary $t\ge t_0$ 
	and $\ell\ge 1$ 
	one deletes deletes at most $\delta_3t^3\ell^3$ edges from a 
	$\bigl(\pird_\star(F)+\tfrac17\eps, \star\bigr)$-dense
	reduced hypergraph with index set $[t]$ whose vertex classes have size $\ell$, 
	then the resulting reduced hypergraph contains a reduced image of $F$. With this choice 
	of $\delta_3$, $\delta_2(\cdot)$, and $t_0$ we appeal to the regularity lemma, 
	thus getting an integer $T_0$. Finally, we set 
	\[
		n_0=2T_0N(T_0)
		\quad \text{ and } \quad
		\eta=\frac{\eps}{56T_0^5}\,.
	\]

	Now we contend that every $(\pird_\star(F)+\eps, \eta, \star)$-dense hypergraph $H$
	on $n\ge n_0$ vertices has a subhypergraph isomorphic to $F$, which clearly implies 
	the desired result. 
	
	Suppose that the regularity lemma applied to $H$ yields the integers 
	$t\in [t_0, T_0]$ and $\ell\leq T_0$, the vertex partition 
	$V(H)=V_0\dcup V_1\dcup\dots\dcup V_t$ and for $1\leq i<j\leq t$ the
	pair partition
	\[
		\cP^{ij}=\{P^{ij}_\alpha=(V_i\dcup V_j,E^{ij}_\alpha)\colon\, 1\leq \alpha \leq \l\}
	\]
	of $K(V_i,V_j)$ such that~\ref{it:RL1},~\ref{it:RL2}, and~\ref{it:RL3} hold.
	
	This situation gives rise to two reduced hypergraphs $\ccA$ and $\ccB$
	with index set $[t]$ and vertex classes $\cP^{ij}$ for $ij\in [t]^{(2)}$
	defined as follows. A triple $\{P^{ij}_\alpha, P^{ik}_\beta, P^{jk}_\gamma\}$
	is declared to form an edge of the constituent $\ccA^{ijk}$ if the corresponding triad
	$P^{ijk}_{\alpha\beta\gamma}$ satisfies $d(H|P^{ijk}_{\alpha\beta\gamma})\ge d_3$.
	If in addition $H$ is $\delta_3$-regular with respect to this triad, then we put this 
	edge into $\ccB^{ijk}$ as well. We shall verify later that 
	\begin{equation}\label{eq:52}
		\ccA \text{ is $\bigl(\pird_\star(F)+\tfrac17\eps, \star\bigr)$-dense}.
	\end{equation}

	Based on this fact, the argument can be completed as follows. 
	By Theorem~\ref{thm:RL}\ref{it:RL3} we have 
	\[
		|E(\ccA)\sm E(\ccB)|\le \delta_3t^3\ell^3\,, 
	\]
	so due to our choice of $t_0$ according to Proposition~\ref{prop:irregular} there is a 
	reduced map $(\lambda, \phi)$ from~$F$ to $\ccB$. Now the embedding lemma applies to 
	$I=[t]$, $\lambda$, the sets $V_i$ for $i\in I$, and the bipartite graphs called 
	$\phi(uv)\in \cP^{\lambda(u)\lambda(v)}$ here playing the r\^{o}les of  
	$P_{uv}$ there. The lower bound imposed there on the sets $V_i$
	follows from 
	\[
		|V_i|=\frac{|V|-|V_0|}{t}\ge \frac{(1-\delta_3)n}{T_0}\ge \frac{n_0}{2T_0}
		= N(T_0)\ge N(\ell)\,,
	\]
	for every $i\in [t]$.
	Moreover, $H$ satisfies the last two bullets of Theorem~\ref{thm:EL} by 
	Definition~\ref{dfn:rm}\ref{it:rm3} and the construction of $\ccB$. 
	So altogether we obtain indeed $F\subseteq H$ and it remains to establish~\eqref{eq:52}.	
	
	A key observation towards this goal is that for $M=|V_1|=\ldots=|V_t|$
	every triad	spans at most $\bigl(\ell^{-3}+3\delta_2(\ell)\bigr)M^3$
	triangles due to the triangle counting lemma, and because of 
	$\delta_2(\ell)\le \frac1{21}\eps\ell^{-3}$ this is turn at most 
	$(1+\tfrac17\eps)\ell^{-3}M^3$. So by our choice of $d_3$ a triad that does not correspond
	to an edge of $\ccA$ can accomodate at most $\tfrac17\eps(1+\tfrac17\eps)\ell^{-3}M^3$
	edges of $H$.
	
	Furthermore, it will be helpful to be aware that 
	our choice of $\eta$ guarantees
	\[
		\eta n^3 = \frac{\eps}{7T_0^2}\left(\frac n{2T_0}\right)^3 
		\le \frac{\eps M^3}{7\ell^2}\,.
	\]

	From now on we treat the three possibilities 
	for $\star$ separately. 
	
	\medskip
	
	{\it \hskip2em First Case. $\star=\vvv$}
	
	\smallskip
	
	Given any three distinct indices $i, j, k\in [t]$ we need to prove 
	$|E(\ccA^{ijk})|\ge \bigl(\pird_{\vvv}(F)+\frac17\eps\bigr)\ell^3$. 
	Applying the assumption that $H$ is 
	$(\pird_{\vvv}(F)+\eps, \eta, \vvv)$-dense to $V_i$, $V_j$, and $V_k$ we obtain
	\[
		|E_{\vvv}(V_i, V_j, V_k)|\ge \bigl(\pird_{\vvv}(F)+\eps\bigr)M^3-\eta n^3
		\ge 
		\bigl(\pird_{\vvv}(F)+\tfrac 67\eps\bigr)M^3\,.
	\]
	Counting the edges of the left side according to the triad to which 
	they belong we obtain 
	\[
		 \bigl(\pird_{\vvv}(F)+\tfrac 67\eps\bigr)M^3
		 \le 
		 \bigl(|E(\ccA^{ijk})|+\tfrac17\eps\ell^3\bigr)
		 	\bigl(1+\tfrac17\eps\bigr)\ell^{-3}M^3\,.
	\]
	Owing to 
	\[
		\bigl(\pird_{\vvv}(F)+\tfrac 47\eps\bigr)\bigl(1+\tfrac17\eps\bigr)
		\le 
		\pird_{\vvv}(F)+\tfrac 67\eps
	\]
	this yields 
	\[
		\bigl(\pird_{\vvv}(F)+\tfrac37\eps\bigr)\ell^3
		\le 	
		|E(\ccA^{ijk})|\,,
	\]
	which is more than required.
	
	\goodbreak
	\medskip
	
	{\it \hskip2em Second Case. $\star=\ev$}
	
	\smallskip
	
	Consider three distinct indices $i, j, k\in [t]$, a bipartite graph 
	$P^{jk}_\gamma\in \cP^{jk}$, and its neighbourhood 
	\[
		N=\bigl\{(P^{ij}_\alpha, P^{ik}_\beta)\in \cP^{ij}\times \cP^{ik}\colon
			\{P^{ij}_\alpha, P^{ik}_\beta, P^{jk}_\gamma\}\in E(\ccA^{ijk})\bigr\}
	\]
	in the constituent $\ccA^{ijk}$. Observe that 
	\[
		|E(P^{jk}_\gamma)| \ge \bigl(\ell^{-1} -\delta_2(\ell)\bigr)M^2
		\ge \bigl(1-\tfrac 17\eps\bigr)\ell^{-1}M^2\,.
	\]
	Since $H$ is $(\pird_{\ev}(F)+\eps, \eta, \ev)$-dense,
	this yields  
	\begin{align*}
		|E_{\ev}(V_i, P^{jk}_\gamma)|
		&\ge 
		\bigl(\pird_{\ev}(F)+\eps\bigr)M|E(P^{jk}_\gamma)|-\eta n^3 \\
		&\ge 
		\bigl(\pird_{\ev}(F)+\eps\bigr)(1-\tfrac 17\eps)\ell^{-1}M^3
			-\tfrac17\eps\ell^{-1}M^3 \\
		&\ge
		\bigl(\pird_{\ev}(F)+\tfrac47\eps\bigr)\ell^{-1}M^3\,,
	\end{align*}
	where we have identified $P^{jk}_\gamma$ in the natural way with a subset of 
	$V_j\times V_k$.
	As in the previous case this leads to 
	\[
		\bigl(\pird_{\ev}(F)+\tfrac47\eps\bigr)\ell^{-1}M^3
		\le
		\bigl(N+\tfrac17\eps\ell^2\bigr)
		 	\bigl(1+\tfrac17\eps\bigr)\ell^{-3}M^3\,,
	\]		
	which in turn implies 
	\[
		\bigl(\pird_{\ev}(F)+\tfrac17\eps\bigr)\ell^{2}	
		\le
		N\,.
	\]
	Thus $\ccA$ is indeed $\bigl(\pird_{\ev}(F)+\tfrac17\eps, \ev\bigr)$-dense.
	
	\goodbreak
	\medskip
	
	{\it \hskip2em Third Case. $\star=\ee$}
	
	\smallskip
	
	This time let three distinct indices $i, j, k\in [t]$ as well as two bipartite graphs
	$P^{ij}_\alpha\in \cP^{ij}$ and $P^{jk}_\gamma\in\cP^{jk}$ be given, which we identify
	with the corresponding subsets of $V_i\times V_j$ and~$V_j\times V_k$, respectively.
	The graph counting lemma implies 
	\[
		\cK_{\ee}(P^{ij}_\alpha, P^{jk}_\gamma)
		\ge 
		\bigl(\ell^{-2}-2\delta_2(\ell)\bigr)M^3
		\ge
		 \bigl(1-\tfrac 17\eps\bigr)\ell^{-2}M^3
	\]
	and it follows from $H$ being $(\pird_{\ee}(F)+\eps, \eta, \ee)$-dense 
	that 
	\begin{align*}
		\big|E_{\ee}(P^{ij}_\alpha, P^{jk}_\gamma)\big|
		&\ge
		\bigl(\pird_{\ee}(F)+\eps\bigr)\big|\cK_{\ee}(P^{ij}_\alpha, P^{jk}_\gamma)\big|
		-\eta n^3 \\
		&\ge
		\bigl(\pird_{\ee}(F)+\eps\bigr)\bigl(1-\tfrac 17\eps\bigr)\ell^{-2}M^3	
		-\tfrac17\eps\ell^{-2}M^3 \\
		& \ge
		\bigl(\pird_{\ee}(F)+\tfrac47\eps\bigr)\ell^{-2}M^3
	\end{align*}
	Regarding the common neighbourhood
	\[
		J=\bigl\{P^{ik}_\beta\in  \cP^{ik}\colon
				(P^{ij}_\alpha, P^{ik}_\beta, P^{jk}_\gamma)\in E(\ccA^{ijk})\bigr\}
	\]
	this tells us
	\[	
		\bigl(\pird_{\ee}(F)+\tfrac47\eps\bigr)\ell^{-2}M^3
		\le
		\bigl(|J|+\tfrac17\ell\eps\bigr)\bigl(1+\tfrac17\eps\bigr)\ell^{-3}M^3\,,
	\]
	which yields
	\[
		\bigl(\pird_{\ee}(F)+\tfrac17\eps\bigr)\ell \le |J|\,,
	\]
	as desired.
\end{proof}

\section{More on tetrahedra}
              
In order to illustrate how Theorem~\ref{thm:33} can be applied we conclude this article
by sketching a proof of $\piee(K^{(3)}_4)=0$. This result forms the first interesting 
case of Theorem~\ref{thm:fortress} and the reader seeking further information or more details 
is referred to~\cite{RRS-d}.

Given $\eps>0$ we want to show that every $(\eps, \ee)$-dense reduced hypergraph 
with sufficiently many indices contains the reduced image of a tetrahedron. Let $\ccA$
be such a reduced hypergraph with index set $I$, vertex classes $\cP^{ij}$, 
and constituents $\ccA^{ijk}$. Write $I$ as a disjoint union $I=X\dcup Y$, 
where $|X|>\tfrac 1\eps$ and $Y$ is much larger then $X$. 

The first step is to assign to every pair $(x, y)\in X\times Y$ an arbitrary 
vertex $P^{xy}\in \cP^{xy}$. 

Next we look at two distinct vertices $x, x'\in X$. 
For every $y\in Y$ the common neighbourhood of $P^{xy}$ and $P^{x'y}$ in the 
constituent $\ccA^{xx'y}$ contains, by our hypothesis on $\ccA$, at least $\eps|\cP^{xx'}|$
vertices. Thus, by double counting, we may fix a vertex $P^{xx'}\in \cP^{xx'}$ belonging to 
this neighbourhood for at least~$\eps|Y|$ many choices of $y\in Y$. In other words, we may 
shrink~$Y$ by a factor of no more than $\eps$ to a subset $Y'$ such that $P^{xx'}P^{xy}P^{x'y}$ 
is edge of $\ccA$ for every $y\in Y'$. 

This argument can be applied iteratively to all pairs of vertices 
in $X$. That is, we enumerate all pairs in $X^{(2)}$ and when processing a pair 
in the list we select a vertex from the corresponding vertex class and shrink the subset
of $Y$ under current consideration by a further factor of $\eps$. When this procedure 
ends, we have chosen for every pair $xx'\in X^{(2)}$ a vertex~$P^{xx'}\in \cP^{xx'}$. 
Moreover, if $Y^*$ denotes the subset of $Y$ that has survived through all stages, then 
$\{P^{xx'}, P^{xy}, P^{x'y}\}\in E(\ccA^{xx'y})$ holds for all distinct $x, x'\in X$ 
and all $y\in Y^*$.   
  
By starting with a sufficiently large set $Y$ we can ensure that $|Y^*|\ge 2$.
Pick once and for all two distinct indices $y, y'\in Y^*$. 
Reversing the r\^{o}les of $X$ and $Y$ 
we may now select a suitable vertex~$P^{yy'}$ in~$\cP^{yy'}$ and shrink $X$ in the same way 
as above to a set $X^*$ with $|X^*|\ge \eps|X|$ such that 
$P^{xy}P^{xy'}P^{yy'}\in E(\ccA^{xyy'})$ holds for all $x\in X^*$. 
Due to $|X|>\tfrac 1\eps$ there will be at least two survivors $x$ and $x'$ in $X^*$.
 
Now the four indices $x$, $x'$, $y$, and $y'$ form together with the six vertices
$P^{xx'}$, $P^{xy}$, $P^{xy'}$, $P^{x'y}$, $P^{x'y'}$, and $P^{yy'}$ the desired 
reduced image of a tetrahedron in $\ccA$.

It should be clear that the same argument also establishes $\pi_{\ee}(B)=0$
for every bipartite hypergraph $B$. There are, however, many further hypergraphs
whose $\ee$-Tur\'an-density vanishes. For instance, as a consequence of Theorem~\ref{thm:zero}
the Fano plane $\ccF$ satisfies~${\pivvv(\ccF)=0}$ and, hence, also~$\piee(\ccF)=0$.  
We shall return to the rather subtle problem of characterising the set $\{F\colon \piee(F)=0\}$
at another occasion.  
       
\subsection*{Acknowledgement} For numerous reasons going much beyond our 
collaboration~\cites{RRS-a, RRS-c, RRS-d, RRS-e, RRS-zero} my indebtedness 
to Vojt\v{e}ch R\"{o}dl and Mathias Schacht is extremely great.

\end{document}